\documentclass[british,11pt]{amsart}
\pdfoutput=1

\usepackage{amsmath, amsfonts, amsthm, amssymb, multicol}
\usepackage{graphicx}
\usepackage{float}
\usepackage{verbatim}

\usepackage{tikz}
\usepackage{verbatim}

 \usepackage[foot]{amsaddr}

\makeatletter
\@namedef{subjclassname@2020}{
  \textup{2020} Mathematics Subject Classification}
\makeatother

\numberwithin{equation}{section}

\allowdisplaybreaks

\calclayout

\usepackage[T1]{fontenc}
\usepackage[utf8]{inputenc}
\usepackage{babel}
\usepackage{csquotes}

\usepackage{mathtools}
\usepackage{microtype}

\newtheorem{thm}{Theorem}[section]
\newtheorem*{theorem*}{Theorem}
\newtheorem{lma}[thm]{Lemma}
\newtheorem{cor}[thm]{Corollary}
\newtheorem{defn}[thm]{Definition}

\newtheorem{prop}[thm]{Proposition}

\newtheorem{ques}[thm]{Question}

\newtheorem{claim}[thm]{Claim}

\usepackage[backend=biber,backref=true,doi=false,
url=false,isbn=false,date=year,sorting=nyt,
giveninits=true,maxbibnames=99,mincrossrefs=9,
sortcites,hyperref=true]{biblatex}
\renewbibmacro{in:}{%
  \ifentrytype{article}{}{\printtext{\bibstring{in}\intitlepunct}}}
\DeclareFieldFormat[article,periodical]{volume}{\mkbibbold{#1}}%
\DeclareFieldFormat[article,periodical,unpublished]{citetitle}{#1}
\DeclareFieldFormat[article,periodical,unpublished]{title}{#1}
\DeclareFieldFormat{pages}{#1}
\DeclareFieldFormat[inproceedings]{title}{#1\isdot}
\DeclareFieldFormat[misc]{title}{#1\isdot}
\DeclareFieldFormat[misc,article]{note}{#1\addcomma}

\AtEveryBibitem{
  \clearfield{number}}
\usepackage{hyperref}

\newcommand{\R}{\mathbb{R}}
\newcommand{\N}{\mathbb{N}}
\newcommand{\Z}{\mathbb{Z}}
\newcommand{\Rd}{\mathbb{R}^d}
\newcommand{\ubd}{\overline{\dim}_{\mathrm{B}}}

\newcommand{\uid}{\overline{\dim}_{\,\theta}}

\newcommand{\asd}{\dim_{\mathrm{A}}}
\newcommand{\asp}{\dim_{\mathrm{A}}^\theta}

\newcommand{\uasp}{\overline{\dim}_\mathrm{A}^\theta}
\newcommand{\qad}{\dim_\mathrm{qA}}

\renewcommand{\epsilon}{\varepsilon}

\renewcommand{\geq}{\geqslant}
\renewcommand{\leq}{\leqslant}

\title{Assouad type dimensions of infinitely generated self-conformal sets}

\author{Amlan Banaji}
\address{Amlan Banaji, Department of Mathematical Sciences, Loughborough University, Loughborough, LE11 3TU, United Kingdom}
\email{A.F.Banaji@lboro.ac.uk}

\author{Jonathan M. Fraser}
\address{Jonathan M. Fraser, School of Mathematics and Statistics, University of St Andrews, St Andrews, KY16 9SS, United Kingdom}
\email{jmf32@st-andrews.ac.uk}

\begin{document}

\date{}

\begin{abstract}
We study the dimension theory of limit sets of iterated function systems consisting of a countably infinite number of conformal contractions. Our focus is on the Assouad type dimensions, which give information about the local structure of sets. Under natural separation conditions, we prove a formula for the Assouad dimension and prove sharp bounds for the Assouad spectrum in terms of the Hausdorff dimension of the limit set and dimensions of the set of fixed points of the contractions. The Assouad spectra of the family of examples which we use to show that the bounds are sharp display interesting behaviour, such as having two phase transitions. Our results apply in particular to sets of real or complex numbers which have continued fraction expansions with restricted entries, and to certain parabolic attractors. \vspace{-.3cm}
\end{abstract}



\maketitle

\tableofcontents

\section{Introduction}

\subsection{Background}

An iterated function system (IFS) is a finite set of contractions $\{ S_i \colon X \to X\}_{i \in I}$, where $X$ is a closed subset of Euclidean space. For each IFS there is an associated limit set (or attractor) $F$ satisfying $F = \cup_{i \in I} S_i(F)$, which will typically be fractal in nature. IFSs and the dimension theory of the associated limit sets have been studied extensively since Hutchinson's important paper~\cite{Hutchinson1981}. In a seminal 1996 paper~\cite{Mauldin1996}, Mauldin and Urbański extended the theory to infinite conformal iterated function systems (CIFSs, defined in Definition~\ref{d:cifs}) consisting of countably many conformal contractions. The maps are assumed to be sufficiently separated, with the contraction ratios uniformly bounded above by some $\xi < 1$. 
One well-studied family of sets which are generated by CIFSs are sets of numbers which have continued fraction expansions with restricted entries. 

There are several notions of fractal dimension which give information about the scaling properties of sets. For background on the Hausdorff and box dimensions, which are defined using global covers of the set, we refer the reader to~\cite{Falconer2014}. In this paper, however, we will focus more on the Assouad type dimensions, which are defined using local covers. The Assouad dimension was introduced in~\cite{Assouad1977thesis} and has been the subject of intensive study in the contexts of embedding theory (see~\cite{Mackay2010,Robinson2011assouad}) and fractal geometry (see~\cite{Fraser2020book}) in recent years. Fraser and Yu~\cite{Fraser2018-2} introduced a parameterised family of dimensions, called the \emph{Assouad spectrum}, which lie between box and Assouad dimension and will be a key focus of this paper. These dimensions are defined in Section~\ref{s:notation}.

The dimension theory of infinite IFSs has been studied extensively, for example in~\cite{Urbanski2022,Chousionis2020,Banaji2021,Mauldin1996,Mauldin1999,Mauldin1995:Fractals1995}.
There are many similarities between finite and infinite iterated function systems, but also several differences. One notable difference is that the Hausdorff, box and Assouad dimensions coincide for the limit set of any finite CIFS, but can differ for infinite CIFSs, indicating that the presence of infinitely many maps can cause the limit set to have greater inhomogeneity. Mauldin and Urbański showed in~\cite[Theorem~3.15]{Mauldin1996} that for an (infinite) CIFS the Hausdorff dimension $\dim_{\mathrm H} F$ can be determined from a topological pressure function (defined in~\eqref{e:MUpressure}). 
For notions of dimension $\dim$ which are not countably stable, a na\"ive guess for a formula for the dimension of the limit set might be $\dim F = \max\{\dim_{\mathrm H} F,\dim \{ S_i(x) \} \}$, where $x$ is any given point in the limit set. Mauldin and Urbański~\cite[Theorem~2.11]{Mauldin1999} proved that the upper box dimension does indeed satisfy this formula. As the main result of~\cite{Banaji2021}, we proved that the upper \emph{intermediate dimensions} $\uid$ also satisfy the na\"ive formula, that is, $\uid F = \max\{\dim_{\mathrm H} F,\uid \{ S_i(x) \} \}$ for all $\theta \in [0,1]$. The intermediate dimensions are a family of dimensions (introduced in~\cite{Falconer2020} and studied further in~\cite{Banaji2022moran,Banaji2021bedford,
Burrell2021-2,Banaji2023gen}) which are parameterised by $\theta \in [0,1]$ and interpolate between the Hausdorff and box dimensions.  
We provided some (often counterintuitive) applications to the dimensions of projections, fractional Brownian images, and general H\"older images. We refer the reader to~\cite{Fraser2021-1} for the general area of `dimension interpolation', which includes both the intermediate dimensions and Assouad spectrum.
One of the most interesting features of the analysis in this paper is that the na\"ive  formula does not always hold for the Assouad spectrum, which instead satisfies two bounds which can be sharp in general. 
In~\cite{Banaji2021}, we also proved bounds for the Hausdorff, box and intermediate dimensions of the limit set without assuming conformality or separation conditions. However, the Assouad dimension can be particularly sensitive to separation conditions even in the case of finite IFSs (see~\cite[Section~7.2]{Fraser2020book}), and in this paper we assume conformality and appropriate separation conditions throughout.

\subsection{Structure of paper and summary of results}

In Section~\ref{s:setting} we introduce notation and define conformal iterated function systems (CIFSs) and their limit sets. We prove geometric estimates for CIFSs which will be useful when proving dimension results later. 
We also define the notions of dimension we will work with, in particular the Hausdorff, upper box and Assouad dimensions (denoted $\dim_{\mathrm H}$, $\ubd$ and $\asd$ respectively), and the Assouad and upper Assouad spectra at $\theta \in (0,1)$ (denoted $\asp$ an $\uasp$ respectively). 

In Section~\ref{s:asd} we prove that under an appropriate separation condition the Assouad dimension of the limit set $F$ satisfies the expected formula 
\[ \asd F = \max\{\dim_{\mathrm H} F, \asd P\},\]
where $P$ is the set of fixed points of the contractions. 

In Section~\ref{s:asp} we prove one of our main results, namely that the Assouad spectrum of the limit set is monotonic in $\theta$ and satisfies the bounds 
\begin{equation}\label{e:introbounds}
 \max\{\dim_{\mathrm H} F,\uasp P \} \leq \asp F \leq \max_{\phi \in [\theta,1]} f(\theta,\phi).  
 \end{equation}
Here, for $\theta \in (0,1)$ and $\phi \in [\theta,1]$, the function $f$ is defined by 
\[ f(\theta,\phi) \coloneqq \frac{(\phi^{-1} - 1) \overline{\dim}_\mathrm{A}^\phi P + (\theta^{-1} - \phi^{-1}) \ubd F}{\theta^{-1} - 1}\] 
and can be thought of as an appropriately weighted average of the Assouad spectrum of $P$ and the upper box dimension of $F$. As with Mauldin and Urbański's proof for the upper box dimension, our proof uses an induction argument to show that the existence of efficient covers at larger scales implies the existence of efficient covers at smaller scales. We also prove several consequences of these bounds, showing for instance that if the Assouad spectrum of the set of fixed points satisfies a form (which we call the three-parameter form) that is often observed, then the Assouad spectrum of the limit set must too. 

In Section~\ref{s:sharp} we provide a family of examples which show that the bounds in~\eqref{e:introbounds} are sharp in general. The Assouad spectra of the family of examples which we use to do so display interesting behaviour not seen previously in dynamically generated sets. In particular, the spectra can have two phase transitions, and need not satisfy the three-parameter form. Figure~\ref{f:sharp} shows the graphs of the Assouad spectra of a selection of these sets. 

Finally, in Section~\ref{s:ctdfracsect} we turn our attention to the Assouad spectra of a well-studied class of infinitely generated self-conformal sets, namely sets of numbers which have continued fraction expansions with restricted entries. Many of the interesting properties of the Assouad spectra witnessed in Section~\ref{s:sharp} also hold in this setting. We calculate a precise formula for the Assouad spectrum of the full complex continued fraction set in terms of its Hausdorff dimension, which has in turn been estimated in~\cite{Mauldin1996,Gardner1983,Priyadarshi2016,Falk2018}. We also investigate parabolic IFSs by applying our methods to the `induced' uniformly contracting CIFS. This in particular allows us to calculate the Assouad spectrum of sets of numbers generated by backwards continued fraction expansions.

\section{Setting and preliminaries}\label{s:setting}

\subsection{Notation and notions of dimension}\label{s:notation}
We denote the natural logarithm by $\log$, the cardinality of a set using $\#$, the (Euclidean) diameter of a subset of $\Rd$ by $|\cdot|$, and $d$-dimensional Lebesgue measure on $\Rd$ by $\mathcal{L}^d$. The symbol $\N$ will denote $\{1,2,3,\ldots\}$. In Sections~\ref{s:sharp} and~\ref{s:ctdfracsect} it will be convenient to write $Y \lesssim Z$ (or $Y \gtrsim Z$) to mean $Y \leq CZ$ (or $Y \geq CZ$, respectively). If $Y \lesssim Z$ and $Y \gtrsim Z$ then we write $Y \approx Z$. We will specify which parameters the implicit constant $C>0$ is allowed to depend on. The symbol $||\cdot||$ will denote either the Euclidean norm on $\Rd$ or the supremum norm of a continuous function, depending on context. 
We write 
\[ B(x,r) \coloneqq \{ \, y \in \Rd : ||y-x|| < r \, \}\ \] 
for the open ball of radius $r>0$ centred at $x \in \Rd$. 
All the sets we consider will be non-empty, bounded subsets of Euclidean space. 
For $F \subseteq \Rd$ let $N_r(F)$ be the smallest number of balls of radius $r$ needed to cover $F$. 
 For $d \in \N$ and $r \geq 1$, we denote by $A_{d,r} \in \N$ the smallest integer such that for all $U \subset \Rd$ there exist $U_1,\ldots,U_{A_{d,r}} \subseteq \Rd$, each of diameter $|U|/r$, such that 
 \begin{equation}\label{e:doublingconst}
 U \subseteq \bigcup_{k=1}^{A_{d,r}} U_k.
 \end{equation}
Let $F$ be non-empty, bounded subset of $\Rd$ with the Euclidean metric. The \emph{Hausdorff dimension} is defined by 
\begin{equation}\label{e:hausdorffdef}
\begin{aligned} \dim_\mathrm{H} F = \inf \{ \, s \geq 0 : &\mbox{ for all } \epsilon >0 \mbox{ there exists a finite or countable cover } \\*
& \{U_1,U_2,\ldots\} \mbox{ of } F \mbox{ such that } \sum_i |U_i|^s \leq \epsilon \,\},
\end{aligned}
\end{equation}%
and the \emph{upper box dimension} is 
\[ \ubd F \coloneqq \limsup_{r \to 0^+} \frac{\log N_r(F)}{-\log r}.\]%
Turning now to notions of dimension which describe the local structure of sets, the \emph{Assouad dimension} is defined by 
\begin{multline}\label{e:assouaddef}
      \dim_\mathrm{A} F = \inf\left\{ \, 
      \alpha : \mbox{ there exists }C>0\mbox{ such that for all } x \in F \mbox{ and } 
    \right. \\* 
    \qquad \left. 0<r<R, \mbox{ we have } N_r(B(x,R)\cap F) \leq C(R/r)^\alpha \, \right\}. 
    \end{multline}
For $\theta \in (0,1)$, the \emph{Assouad spectrum} of $F$ at $\theta$ is defined by fixing the scales $R = r^\theta$: 
   \begin{multline*}
      \asp F = \inf\left\{ \, 
      \alpha : \mbox{ there exists }C>0\mbox{ such that for all } x \in F \mbox{ and } 
    \right. \\* 
    \qquad \left. 0< R \leq 1, \mbox{ we have } N_{R^{1/\theta}}(B(x,R)\cap F) \leq C R^{\alpha(1 - 1/\theta)} \, \right\}. 
    \end{multline*}
    Clearly $\ubd F \leq \asp F \leq \asd F$ for all $\theta \in (0,1)$. 
    The Assouad spectrum is not always monotonic in $\theta$ (see~\cite{Fraser2018-2}). The \emph{upper Assouad spectrum} at $\theta$, however, \emph{is} monotonic, and is defined by 
      \begin{multline*}
      \uasp F = \inf\left\{\, 
      \alpha : \mbox{ there exists }C>0\mbox{ such that for all } x \in F \mbox{ and } 
    \right. \\* 
    \qquad \left. 0<r\leq R^{1/\theta} \leq R \leq 1, \mbox{ we have } N_r(B(x,R)\cap F) \leq C(R/r)^\alpha \, \right\}. 
    \end{multline*}
    The Assouad spectrum was introduced in~\cite{Fraser2018-2} and has been calculated for various families of fractals in~\cite{Fraser2018assouadfamilies,FraserPreprintsullivan,Burrell2020-1} and other works. Recently, Rutar~\cite{Rutar2022assouad} has given a complete description of the attainable forms of Assouad spectra of sets, showing that a wide variety of behaviour is possible in general. 
    The Assouad spectrum can be used to give information about dimensions of orthogonal projections of sets~\cite{Falconer2021projections}, and has applications related to spherical maximal functions~\cite{Roos2020assouadspectrum} and conformal geometry~\cite{Garitsis2022conformal}. 
The \emph{quasi-Assouad dimension}, introduced in~\cite{Lu2016}, can be defined by  
\begin{equation}\label{e:quasiassouaddef}
 \qad F \coloneqq \lim_{\theta \to 1^-} \asp F,
 \end{equation}
or equivalently $\qad F \coloneqq \lim_{\theta \to 1^-} \uasp F$ (see~\cite[Corollary~3.3.7]{Fraser2020book}). 
 We always have 
 \[ \dim_{\mathrm H} F \leq \ubd F \leq \asp F \leq \uasp F \leq \qad F \leq \asd F,\] 
  and all inequalities can be strict in general. 
  We sometimes write $\dim_{\mathrm A}^1$ or $\overline{\dim}_{\mathrm A}^1$ to mean the quasi-Assouad dimension, and since $\uasp F \to \ubd F$ as $\theta \to 0^+$, we sometimes write $\dim_{\mathrm A}^{0} F$ or $\overline{\dim}_{\mathrm A}^0 F$ to mean the upper box dimension of $F$. 
  The Assouad spectrum and upper Assouad spectrum are continuous in $\theta \in (0,1)$, see~\cite{Fraser2018-2,Fraser2019-3}. In~\cite{Fraser2019-3}, Fraser et al. show that we always have $\overline{\dim}_\mathrm{A}^\theta F = \sup_{\theta' \in (0,\theta]} \dim_\mathrm{A}^{\theta'} F$. Therefore if $\theta \mapsto \asp F$ is monotonic (as holds for many sets of interest, see Lemma~\ref{l:monotoniclemma}) then we can use the Assouad spectrum and upper Assouad spectrum interchangeably. 
For a set $F \subseteq \Rd$ we define the quantity 
\begin{equation}\label{e:definephasetransition}
 \rho = \rho_F \coloneqq \inf \{ \, \theta \in [0,1] : \asp F = \qad F \, \} = \inf \{ \, \theta \in [0,1] : \uasp F = \qad F \, \}, 
 \end{equation}
 which will often represent a phase transition in the Assouad spectrum of $F$. 
The two expressions are equal by~\cite[Corollary~3.3.3]{Fraser2020book}.

\subsection{Conformal iterated function systems}
We follow the setup of Mauldin and Urbański~\cite{Mauldin1996,Mauldin1999}, for which we need to introduce some notation. %
Given a countable index set $I$, define $I_0 \coloneqq \{\varnothing\}$ and $I^* \coloneqq \bigcup_{i=1}^\infty I^n$. We call elements of $I^*$ \emph{finite words} and elements of $I^\N$ \emph{infinite words}. We usually denote words by the letter $w$, and we write $w = i_1 \cdots i_n$ and $w=i_1i_2\ldots$ instead of $w= (i_1,\ldots,i_n)$ and $w=(i_1,i_2,\ldots)$ respectively. We say that a word in $I^n$ has \emph{length} $n$, and an infinite word has \emph{length} $\infty$. If $w \in I^* \cup I^\N$ and $n \in \N$ does not exceed the length of $w$ then we write $w|_n \coloneqq w_1 \ldots w_n \in I^n$, and $w|_0 \coloneqq \varnothing$. If $w \in I_0 \cup I^* \cup I^\N$ and $v \in I_0 \cup I^*$ then we say that $v$ is a \emph{prefix} of $w$ if there exists $n \in \{0,1,2,\ldots\}$ such that $v = w|_n$.

In the following definition, the assumption that the contractions are conformal maps (meaning that they locally preserve angles) is especially important. In one dimension, conformal maps are functions with a non-vanishing H\"older continuous derivative, in two dimensions they are holomorphic functions with non-vanishing derivative, and in three dimensions they are M\"obius maps (by Liouville's theorem). %
\begin{defn}\label{d:cifs}
Let $d \in \N$ and let $X$ be a compact, connected subset of $\Rd$, equipped with the Euclidean metric. Consider a collection of maps $S_i \colon X \to X$, $i \in I$, where $I$ is a countable index set. This system forms a \emph{conformal iterated function system (CIFS)} if the following additional properties are satisfied: %
\begin{enumerate}

\item\label{i:osc} (Open Set Condition (OSC)) The set $X$ has non-empty interior $U \coloneqq \mathrm{Int}_{\mathbb{R}^d} X$, and $S_i(U) \subset U$ for all $i \in I$, and $S_i(U) \cap S_j(U) = \varnothing$ for all $i,j \in I$ with $i \neq j$. 

\item\label{i:cone} (Cone condition) $\inf_{x \in X} \inf_{r \in (0,1)} \mathcal{L}^d (B(x,r) \cap \mathrm{Int}_{\Rd} X)/r^d > 0$. 

\item\label{i:conformal} (Conformality) There exists an open, bounded, connected subset $V \subset \Rd$ such that $X \subset V$ and such that for each $i \in I$, $S_i$ extends to a $C^{1+\epsilon}$ diffeomorphism $\overline{S_i}$ from $V$ to an open subset of $V$. Moreover, $\overline{S_i}$ is \emph{conformal} on $V$, which means that for all $x \in V$ the differential $\overline{S_i}'|_x$ exists, is non-zero, is a similarity map (so $||\overline{S_i}'|_x (y)|| = ||\overline{S_i}'|_x||\cdot||y||$ for all $y \in \Rd$), and is $\epsilon$-H{\"o}lder continuous in $x$. Furthermore, there exists $\xi \in (0,1)$ such that $||\overline{S_i}'|| <\xi$ for all $i \in I$, where $||\cdot||$ is the supremum norm over $V$. 

\item\label{i:bdp} (Bounded Distortion Property (BDP)) There exists $K>0$ such that $||S_w'|_y|| \leq K||S_w'|_x||$ for all $x,y \in V$ and $w \in I^*$. %

\end{enumerate}
\end{defn}

For $w \in I^n$ we define 
\[ S_w \coloneqq S_{w_1} \circ \cdots \circ S_{w_n},\]  
and we define $S_\varnothing$ to be the identity function on $X$. 
Since $|S_{w|_n}(X)| \leq \xi^n |X|$ by the uniform contractivity, the map 
\[ \pi \colon I^\N \to X, \qquad \pi(w) \coloneqq \bigcap_{n=1}^\infty S_{w|_n}(X) \]
is well-defined and continuous. 
We are interested in the following set, which will typically be fractal in nature. 
\begin{defn}
The \emph{limit set} or \emph{attractor} of a CIFS is defined by
\[ F \coloneqq \pi(I^\N) = \bigcup_{w \in I^\N} \bigcap_{n=1}^\infty S_{w|_n}(X). \]
\end{defn}
For $w \in I^n$ define $F_w = F_{S_w} \coloneqq S_w(F)$ and $X_w = X_{S_w} \coloneqq S_w(X)$. 
Now, $F$ is clearly non-empty and satisfies the relation 
\begin{equation}\label{e:attractor} F = \bigcup_{i \in I} F_i. 
\end{equation}
In fact there are many sets which satisfy~\eqref{e:attractor}, and $F$ is the largest of these by inclusion. If $I$ is finite then $F$ is compact (and is indeed the only non-empty compact set which satisfies~\eqref{e:attractor} by Hutchinson's application of the Banach contraction mapping theorem~\cite{Hutchinson1981}), but if $I$ is infinite then $F$ will not generally be closed. %
For $w \in I_n$ define 
\begin{align}\label{e:definerw}
\begin{split}
r_w &= r_{S_w} \coloneqq \inf_{x,y \in X, x \neq y} \frac{||S_w(x)-S_w(y)||}{||x-y||}; \\
R_w &= R_{S_w} \coloneqq \sup_{x,y \in X, x \neq y} \frac{||S_w(x)-S_w(y)||}{||x-y||},
\end{split}
\end{align}
noting that $0 \leq r_w \leq R_w \leq \xi$. The value $R_w$ is the smallest possible Lipschitz constant for $S_w$. 
Mauldin and Urbański~\cite{Mauldin1996} introduced 
\[ \psi_n(t) \coloneqq \sum_{w \in I^n} ||S_w'||^t \]
for $n \in \N$, $t \in (0,\infty)$, and defined the topological pressure function $\overline{P} \colon (0,\infty) \to [-\infty,\infty]$ by %
\begin{equation}\label{e:MUpressure} \overline{P}(t) \coloneqq \lim_{n \to \infty} \frac{1}{n} \log \psi_n (t). 
\end{equation}
Mauldin and Urbański~\cite[Theorem~3.15]{Mauldin1996} proved that the Hausdorff dimension of the limit set is given by 
\begin{equation}\label{e:hausdorffpressure}
 h \coloneqq \dim_{\mathrm H} F = \inf\{ \, t \geq 0 : \overline{P}(t) \leq 0 \, \}.
 \end{equation}
 Throughout the paper, we reserve the letter $h$ for the quantity in~\eqref{e:hausdorffpressure}. 
 In fact, this formula holds even if the cone condition~\eqref{i:cone} is not assumed, see~\cite[Theorem~19.6.4]{Urbanski2022}. 
 However, the cone condition is used in the proof of~\cite[Lemma~2.11]{Banaji2021}, which is in turn used in the proofs of Theorem~\ref{t:mainasp} below. The condition will in particular be satisfied if $X$ is convex or has smooth boundary. 

 Since the pressure function is decreasing, the \emph{finiteness parameter} of the system 
 \begin{equation}\label{e:finiteness}
  \Theta \coloneqq \inf\{ \, t > 0 : \overline{P}(t) < \infty \, \} \in [0, \infty] 
  \end{equation}
  is always a lower bound for the Hausdorff dimension.

 We will use many properties of CIFSs. It is well known (see~\cite[Lemmas~2.9 and~2.12]{Banaji2021}, for example) that for any CIFS there exists $D \geq 1$ such that for all $w \in I^*$,  
\begin{equation}\label{e:diameterslemma}
 D^{-1} ||S_w'|| \leq r_w \leq R_w \leq D||S_w'||;
\end{equation}
\begin{equation}\label{e:2nddiams}
D^{-1} ||S_w'|| \leq |F_w| \leq |X_w| \leq D||S_w'||. 
\end{equation}
In~\cite[Lemma~2.11]{Banaji2021} we proved that 
there exists $M \in \N$ such that for all $z \in \Rd$ and $r>0$, if $w_1,\ldots,w_l$ are distinct words in $I^*$ such that for all $i,j \in \{1,\ldots,l\}$, $w_i$ is not a prefix of $w_j$, and 
\begin{equation}\label{e:mdefn} B(z,r) \cap S_{w_i}(X) \neq \varnothing \mbox{ and } |S_{w_i}(X)| \geq r/2 \mbox{ for all } i \in \{1,\ldots,l\}, 
\end{equation}
 then $l \leq M$. %
We will use the following consequences of conformality of the maps when proving results about the Assouad-type dimensions. 

\begin{lma}\label{l:finitelymanylarge}
For any CIFS and any $\lambda > 0$ there are only finitely many $w \in I^*$ with $||S_w'|| \geq \lambda$. 
\end{lma}
\begin{proof}
Let $D$ be as in~\eqref{e:diameterslemma},~\eqref{e:2nddiams}. Note that there exists a ball $B(x,r) \subset \mathrm{Int}_{\Rd}(X)$. For $n \in \{1,\ldots,N-1\}$ let $S_n \coloneqq \{ \, w \in I^n : ||S_w'|| \geq \lambda \, \}$. For each $w \in L_n$, $S_w(\mathrm{Int}_{\Rd}(X))$ contains a ball $B_w$ of radius $\lambda r / D$ by~\eqref{e:diameterslemma}. Moreover, if $w, v \in L_n$ are distinct then by the OSC, $B_w \cap B_v = \varnothing$. 
Therefore by a Lebesgue measure argument, $\# L_n < \infty$. But if $N \in \N$ is large enough that $\xi^N < \lambda$, then $\# L_n = 0$ for all $n \geq N$. Thus $\# \{ \, w \in I^* : ||S_w'|| \geq \lambda \, \} = \sum_{n=1}^\infty \# L_n < \infty$. 
\end{proof}

\begin{lma}\label{l:assouadgeo}
There exists $D' > 0$ such that for all $Y \subseteq X$, $w \in I^*$ and $r>0$,  
\[ (D')^{-1} N_r(Y) \leq N_{||S_w'||r}(S_w(Y)) \leq D' N_r(Y);\] 
\[ (D')^{-1} N_r(Y) \leq N_{|X_w|r}(S_w(Y)) \leq  N_{|F_w|r}(S_w(Y)) \leq D' N_r(Y).\] 
\end{lma}%

\begin{proof}
Let $D'$ be the constant $A_{d,2D}$ from~\eqref{e:doublingconst}. Let $Y \subset X$, $w \in I^*$ and $r>0$. For the upper bound note that there exist balls $B_1,\ldots,B_{N_r(Y)}$ of radius $r$ which cover $Y$. Assume without loss of generality that each of these balls intersects $Y$. Then by the upper bound of~\eqref{e:diameterslemma}, for each $j$, 
 \[ |S_w(B_j \cap X)| \leq D||S_w'||\cdot |B_j \cap X| \leq D||S_w'||\cdot |B_j| = D||S_w'||(2r).\] %
The $B_j$ cover $Y$ so $\{S_w(B_j \cap X)\}_j$ covers $S_w(Y)$. For each $j$ we can cover $S_w(B_j \cap X)$ by $D'$ balls of radius $||S_w'||r$, so $N_{||S_w'||r}(S_w(Y)) \leq D' N_r(Y)$. 

For the lower bound, note that there exist balls $B_1',\ldots,B_{N_{||S_w'||r}(S_w(Y))}'$ of radius $||S_w'||r$ which cover $S_w(Y)$, each intersecting $S_w(Y)$. By the lower bound of~\eqref{e:diameterslemma}, for each $k$,  
\begin{align*} 
|S_w^{-1}(B_k' \cap S_w(Y))| &\leq D||S_w'||^{-1} \cdot |B_k' \cap S_w(Y)| \\
&\leq D||S_w'||^{-1} \cdot |B_k'| \\
&= D||S_w'||^{-1} (2||S_w'||r) \\
&= 2Dr. 
\end{align*}
The $\{S_w^{-1}(B_k' \cap S_w(Y))\}_k$ cover $Y$. We can cover each of these sets by $D'$ balls of radius $r$, so $N_r(Y) \leq D' N_{||S_w'||r}(S_w(Y))$, proving the lower bound. In light of~\eqref{e:2nddiams} we may increase $D'$ further to ensure that the second string of inequalities also holds. 
\end{proof}

The following lemma is similar to~\cite[Proposition~2.9]{Mauldin1999} and~\cite[Lemma~3.3]{Banaji2021}. Recall the notation from~\eqref{e:doublingconst}, and recall that $M$ is defined in~\eqref{e:mdefn}. 
\begin{lma}\label{l:samewithinlevelasd}
Consider a CIFS and fix $n \in \N$. If $P$ and $Q$ are both subsets of $\cup_i S_i(X)$ which intersect $S_w(X)$ in exactly one point for each $w \in I^n$ then for all $x \in X$ and $R,r>0$,  
\begin{equation}\label{e:samewithinbound}
 N_r(P \cap B(x,R)) \leq (A_{d,6} + M) N_r(Q \cap B(x,R)). 
 \end{equation}
Moreover, $\asd P = \asd Q$, and $\asp P = \asp Q$ and $\uasp P = \uasp Q$ for all $\theta \in (0,1)$. 
\end{lma}

\begin{proof}
The set of maps corresponding to words of length $n$ forms another CIFS, so we may assume without loss of generality that $n=1$. 
Suppose $x_1,\ldots,x_N \in X$ are such that the balls $B(x_1,r),\ldots,B(x_N,r)$ cover $Q$. For each $j=1,\ldots,N$ let $y_{j,1},\ldots,y_{j,A_{d,6}} \in X$ be such that 
\[ B(x_j,3r) \subseteq \bigcup_{l=1}^{A_{d,6}} B(y_{j,l},r).\]
 Then this union covers $S_i(X)$ for each $i \in I$ such that $|S_i(X)| < r$ and $S_i(X) \cap B(x_j,r) \neq \varnothing$. By~\cite[Lemma~2.11]{Banaji2021} there exist $i_1, \ldots, i_M \in I$, depending on $j$ and not necessarily distinct, such that $S_{i_k}(X) \cap B(x_j,r) \neq \varnothing$ for $k=1,\ldots,M$, and such that if $i \in I \setminus \{i_1,\ldots,i_M\}$ and $|S_i(X)| \geq r$ then $S_i(X) \cap  B(x_j,r) = \varnothing$. If $k=1,\ldots,M$ then we can cover the single element of $P \cap S_{i_k}$ by a ball $B_{j,k}$ of radius $r$. Since $\{B(x_j,r)\}_j$ covers $Q$, 
\[P \subseteq \bigcup_j \left( \bigcup_{l=1}^{A_{d,6}} B(y_{j,l},r) \cup \bigcup_{k=1}^M B_{j.k} \right),\]
proving~\eqref{e:samewithinbound}. It follows immediately that $\asd P = \asd Q$ and $\asp P = \asp Q$ for all $\theta \in (0,1)$. Since $\overline{\dim}_\mathrm{A}^\theta F = \sup_{\phi \in (0,\theta]} \dim_\mathrm{A}^{\phi} F$ for all $F \subseteq \mathbb{R}^d$ (see~\cite{Fraser2019-3}), it also follows that $\uasp P = \uasp Q$. %
\end{proof}

\section{Assouad dimension of the limit set}\label{s:asd}

Theorem~\ref{t:asd} gives a simple formula for the Assouad dimension of the limit set under an additional separation condition, which we assume in this section only. As discussed in Section~\ref{s:further}, we do not know whether the additional assumption can be removed. Recall that $h \coloneqq \dim_{\mathrm H} F$. 

\begin{thm}\label{t:asd}
Consider a CIFS with notation as in Definition~\ref{d:cifs} and assume that $\overline{S_i}(V) \cap \overline{S_j}(V) = \varnothing$ for all distinct $i,j \in I$. Let $P$ be a subset of $\cup_i S_i(X)$ intersecting $S_i(X)$ in exactly one point for each $i$. Then the limit set $F$ satisfies 
\[ \asd F = \max\{h,\asd P\}.\]
\end{thm}
\begin{proof}%
Let $x \in F$ and $0 < r < R < |F|$, and write $B \coloneqq B(x,R)$. We will cover $B \cap F$ efficiently at scale $r$ by noting that by the separation condition, the cylinders with size much greater than $R$ which intersect $B$ must form a nested sequence. Within the deepest of these cylinders, we will count subcylinders of a given size between $r$ and $R$ using the Assouad dimension of $P$ and cover each using the box dimension of $F$, and we use $\asd P$ to cover all cylinders which are smaller than $r$.

Fix $\delta \in(0,\mbox{dist}(X,\Rd \setminus V))$, and let $X'$ be the closed $\delta$-neighbourhood of $X$. 
By the proof of~\cite[Lemma~2.9]{Banaji2021}, we may increase $D$ (which is defined near~\eqref{e:diameterslemma}) so that~\eqref{e:diameterslemma} holds even if the infimum and supremum in the definition of $r_w$ and $R_w$ respectively (see~\eqref{e:definerw}) are taken over $X'$ instead of $X$. 
 First observe that by the assumed separation condition, there is a unique $w \in I^*$ such that $F_w \cap B \neq \varnothing$ and $|F_w| \geq D^2 R / \delta$ and such that if $v \in I^*$ satisfies $F_v \cap B \neq \varnothing$ and $|F_v| \geq D^2 R / \delta$ then $v$ is a subword of $w$. 
Then $B \cap F = B \cap F_w$. 
It is possible that $w$ could be the empty word, in which case recall that $S_{\varnothing}$ is defined to be the identity map. 
If $i \in I$ is such that $S_w^{-1}(B) \cap F_i \neq \varnothing$, then $|F_i| \leq D^3 R / (\delta ||S_w'||)$. 
Let $k_0 \geq 0$ be such that $D^3 R 2^{-(k_0+1)} / \delta < r \leq D^3 R 2^{-k_0} / \delta$. For $0 \leq k \leq k_0$, define 
\[ I_k \coloneqq \{ \, i \in I : F_i \cap S_w^{-1}(B) \neq \varnothing, \, D^3 R 2^{-(k+1)} / ( \delta ||S_w'||)  < |F_i| \leq D^3 R 2^{-k} / (\delta ||S_w'|| ) \, \}. \] 
Let $t > s > \max\{h,\dim_{\mathrm A} P \}$ and let $C>0$ be the constant from the definition of Assouad dimension~\eqref{e:assouaddef} corresponding to exponent $s$. 
Since $t > \ubd F$, we may increase $C$ further so that for all $r' \in (0,|F|]$ we have $N_{r'}(F) \leq C (r')^{-t}$. 
Then 
\begin{align*}
 \# I_k &\leq \# \{ \, i \in I : X_i \subset B(S_w^{-1}(x),2 D^3 R/(\delta ||S_w'||)), \\*
 &\phantom{--------} D^3 R 2^{-(k+1)} /( \delta ||S_w'||) < |F_i| \leq D^3 R/(\delta ||S_w'||) \, \} \\
&\leq M N_{D^3 R 2^{-(k+1)} / (\delta ||S_w'||)} (P \cap B(S_w^{-1}(x),2D^3R/(\delta ||S_w'||)))  \quad \qquad \mbox{(M is from~\eqref{e:mdefn})} \\
&\leq 4 M C 2^{ks}. 
\end{align*}
If $i \in I_k$ then by Lemma~\ref{l:assouadgeo} and~\eqref{e:2nddiams}, 
\[ N_{r/||S_w'||}(F_i) \leq D' N_{r/(||S_w'||\cdot ||S_i'||)}(F) \leq C D' (D^4 R 2^{-k}  / (\delta r))^t = C D' D^{4t} \delta^{-t} 2^{-kt} (R/r)^t.\] 
We also define 
\[ \mathit{SMALL} \coloneqq \{ \, j \in I : F_j \cap B(S_w^{-1}(x),DR/||S_w'||) \neq \varnothing, \, |F_j| \leq D^3 R 2^{-(k_0+1)} / (\delta ||S_w'||) \, \}. \]
Then since $s > \asd P$, recalling that $A_{d,3}$ is from \eqref{e:doublingconst}, 
 \[ N_{r/||S_w'||}( \{ \, F_j : j \in \mathit{SMALL} \, \} ) \leq A_{d,3} C \left(\frac{DR/||S_w'||}{r/||S_w'||}\right)^s = A_{d,3} C D^s (R/r)^s. \] 
Also, $S_w^{-1}(B) \cap F \subseteq B(S_w^{-1}(x),DR/||S_w'||) \cap F$, so by Lemma~\ref{l:assouadgeo}, 
\begin{align*}
N_r(B \cap F) &\leq D' N_{r/||S_w'||} (S_w^{-1}(B) \cap F) \\
&\leq D' \left( \sum_{k = 0}^{k_0} (\# I_k)C D' D^{4t} \delta^{-t} 2^{-kt} (R/r)^t  +   N_{r/||S_w'||}( \{ \, F_j : j \in \mathit{SMALL} \, \} ) \right) \\
&\leq D' \left( 4 M C^2 D' D^{4t} \delta^{-t}  \sum_{k = 0}^\infty 2^{-(t-s)k} + A_{d,3} C D^s \right) (R/r)^t .
\end{align*}
Since $s,t$ were arbitrary, the result follows. 
\end{proof}

The Assouad dimension is related to \emph{porosity}. A set is said to be porous if there exists $\alpha \in (0,1/2)$ such that for all $x \in F$ and $r>0$ there exists $y \in B(x,r)$ such that $B(y,\alpha r) \cap F = \varnothing$. 

\begin{cor}
Suppose a CIFS on $\Rd$ satisfies $\overline{S_i}(V) \cap \overline{S_j}(V) = \varnothing$ whenever $i \neq j$. Let $P$ be a subset of $\cup_i S_i(X)$ intersecting $S_i(X)$ in exactly one point for each $i$, and assume the limit set $F$ satisfies $\dim_{\mathrm H} F < d$. Then $F$ is porous if and only if $P$ is porous. 
\end{cor}
\begin{proof}
A subset of $\Rd$ is porous if and only if its Assouad dimension is less than $d$ (see for example~\cite[Theorem~5.1.5]{Fraser2020book}). Therefore the result follows from Theorem~\ref{t:asd}. 
\end{proof}

\section{Bounds for the Assouad spectrum}\label{s:asp}

\subsection{The bounds}

The following lemma shows that there is a certain uniformity in the definition of the upper Assouad spectrum. 
\begin{lma}\label{l:uniformconst}
For all bounded sets $F \subset \Rd$, $\epsilon > 0$, $0 < \beta < 1$ there exists $R_{F,\beta,\epsilon} \in (0,1)$ such that for all $\theta \in (0,\beta]$ and $x \in F$, if $0<r\leq R^{1/\theta} \leq R \leq R_{F,\beta,\epsilon}$ then 
\[N_r(B(x,R) \cap F) \leq (R/r)^{\uasp F + \epsilon}.\]
\end{lma}

\begin{proof}
Since $\uasp F$ is a continuous function of $\theta$, for all $\theta \in (0,1)$ there exists $\eta = \eta_\theta > 0$ small enough that $\eta < \max\{\theta,1-\theta\}$ and $\overline{\dim}_{\mathrm A}^{\theta + \eta} F < \overline{\dim}_{\mathrm A}^{\theta - \eta} F + \epsilon/2$. There exists $\alpha \in (0,\beta)$ small enough that $\overline{\dim}_{\mathrm A}^\alpha F < \ubd F + \epsilon/2$. Since $[\alpha,\beta]$ is compact, there exists a finite set $\theta_1,\ldots,\theta_n \in [\alpha,\beta]$ such that $[\alpha,\beta] \subset \cup_{i=1}^n (\theta_i - \eta_{\theta_i},\theta_i + \eta_{\theta_i})$. By definition of the upper Assouad spectrum, for each $i$ there exists $C_i > 1$ such that for all $\theta \in (\theta_i - \eta_{\theta_i},\theta_i + \eta_{\theta_i})$, $x \in F$ and $0<r\leq R^{1/\theta} \leq R \leq 1$ (so $r < R^{1/(\theta_i + \eta_{\theta_i})}$), we have 
\[ N_r(B(x,R) \cap F) \leq C_i (R/r)^{\overline{\dim}_{\mathrm A}^{\theta_i - \eta_{\theta_i}} F + \epsilon/2} \leq C_i (R/r)^{\overline{\dim}_{\mathrm A}^{\theta} F + \epsilon/2}. \] 
There exists $C_0 > 1$ such that for all $\theta \in (0,\alpha]$, $x \in F$, $0<r\leq R^{1/\theta} \leq R \leq 1$ (so $r \leq R^{1/\alpha}$), we have 
\[  N_r(B(x,R) \cap F) \leq C_0 (R/r)^{\ubd F + \epsilon/2} \leq C_0 (R/r)^{\uasp F + \epsilon/2}.\] 
Let $C_{\beta,\epsilon} \coloneqq \max_{0 \leq i \leq n} C_i$. 
Choose $R_{F,\beta,\epsilon} \in (0,1)$ small enough that $R_{F,\beta,\epsilon}^{(\beta^{-1} - 1)\epsilon/2} < C_{\beta,\epsilon}^{-1}$. Suppose $0<r\leq R^{1/\theta} \leq R \leq R_{F,\beta,\epsilon} < 1$. Then 
\[ N_r(B(x,R) \cap F) \leq C_{\beta,\epsilon} (R/r)^{\uasp F + \epsilon/2} \leq (R/r)^{\uasp F + \epsilon}, \]
as required. 
\end{proof}

In~\cite[Section~8]{Fraser2018-2} it was shown that the Assouad spectrum of sets is not always monotonic, but the following lemma shows that for limit sets of a CIFS it is in fact monotonic. 
\begin{lma}\label{l:monotoniclemma}
If $F$ is the limit set of a CIFS then the function $\theta \mapsto \asp F$ is increasing in $\theta \in (0,1)$. 
\end{lma}

\begin{proof}
Suppose $0<\theta<\phi<1$. The idea is that if a part of $F$ is difficult to cover at a scale corresponding to $\theta$, then the image of this part of $F$ within a cylinder with an appropriately chosen contraction ratio will be difficult to cover at the scale corresponding to $\phi$. Moreover, we can iterate any given map to choose a cylinder with the desired contraction ratio up to a constant multiple. 

Let $t \in (0,\asp F)$. 
Then there exist sequences $x_n \in F$ and $r_n \in (0,1)$ such that $r_n \to 0$ as $n \to \infty$ and 
\begin{equation}\label{e:assmoneqn}
r_n^{t(1- \theta)} N_{r_n}(F \cap B(x_n,r_n^\theta)) \to \infty \mbox{ as } n \to \infty.
\end{equation} 
Let $S$ be any map in the IFS and define $S^0$ to be the identity function on $V$, and $S^l \coloneqq S \circ \dotsb \circ S$, $l$ times, and $F_l \coloneqq S^l(F)$. Fix $n \in \N$. 
Noting that $||(S^l)'||$ is decreasing in $l$, let $k = k(n)$ be the smallest natural number such that $||(S^k)'|| \leq (Dr_n^{\theta-\phi})^{\frac{1}{\phi - 1}}$, where $D$ is from~\eqref{e:diameterslemma}. Then 
\begin{align*}
N_{||(S^k)'||r_n}&(F \cap B(S^k(x_n),(||(S^k)'||r_n)^\phi)) \\
&\geq  N_{||(S^k)'||r_n}(F \cap B(S^k(x_n),D||(S^k)'||r_n^\theta)) &&\text{by the definition of } k \\
&\geq  N_{||(S^k)'||r_n}(S^k(F \cap B(x_n,r_n^\theta))) &&\text{by \eqref{e:diameterslemma}} \\
&\geq  (D')^{-1} N_{r_n}(F \cap B(x_n,r_n^\theta)) &&\text{by Lemma~\ref{l:assouadgeo}}. 
\end{align*}
It follows that 
\begin{align*}
(||(S^k)'||&r_n)^{t(1-\phi)} N_{||(S^k)'||r_n}(F \cap B(S^k(x_n),(||(S^k)'||r_n)^\phi)) \\
&\geq (||(S^{k-1})'||\cdot ||S'|| K^{-1} r_n)^{t(1-\phi)}(D')^{-1} N_{r_n}(F \cap B(x_n,r_n^\theta)) &&\text{by the BDP} \\
&> (((Dr_n^{\theta-\phi})^{\frac{1}{\phi - 1}}) ||S'|| K^{-1} r_n)^{t(1-\phi)} (D')^{-1} N_{r_n}(F \cap B(x_n,r_n^\theta)) &&\text{by the definition of } k \\
&= D^{-t} ||S'||^{t(1-\phi)} K^{t(\phi-1)}(D')^{-1}  r_n^{t(1- \theta)} N_{r_n}(F \cap B(x_n,r_n^\theta)) \\
&\xrightarrow[]{} \infty \mbox{ as } n \to \infty &&\text{by \eqref{e:assmoneqn}}.
\end{align*}%
Therefore $\dim_{\mathrm{A}}^\phi F \geq t$, and letting $t \to (\asp F)^-$ gives $\dim_{\mathrm{A}}^\phi F \geq \asp F$, as required. 
\end{proof}

Fix an arbitrary $P \subseteq \cup_i S_i(X)$ intersecting $S_i(X)$ in exactly one point for each $i$. For $\theta \in (0,1)$ and $\phi \in [\theta,1]$, we introduce the continuous function 
\begin{equation}\label{e:fdefn}
f(\theta,\phi) \coloneqq \frac{(\phi^{-1} - 1) \overline{\dim}_\mathrm{A}^\phi P + (\theta^{-1} - \phi^{-1}) \ubd F}{\theta^{-1} - 1}.
\end{equation}
By Lemma~\ref{l:samewithinlevelasd}, the function $f(\theta,\phi)$ does not depend on the choice of $P$. 
Recall that $\ubd F = \max\{h,\ubd P\}$ by~\cite[Theorem~3.5]{Banaji2021}. 
Note in particular that $f(\theta,\theta) = \uasp P$ and $f(\theta,1) = \ubd F$. 
It will be important to note that by the continuity of the upper Assouad spectrum, for fixed $\theta \in (0,1)$, the function $\phi \mapsto f(\theta,\phi)$ is continuous, and hence attains a maximum, on the interval $\phi \in [\theta,1]$. The following technical lemma will be used in the proof of Theorem~\ref{t:mainasp}.

\begin{lma}\label{l:fbound}
If $0 < \theta' \leq \theta < 1$ and $\theta' \leq \phi' \leq 1$ then $f(\theta',\phi') \leq \max_{\phi \in [\theta,1]} f(\theta,\phi)$. 
\end{lma}
\begin{proof}
We may assume $\theta' < \theta$. Let $\phi_1 \in [\theta',1]$ be such that $\max_{\phi \in [\theta',1]} f(\theta',\phi) = f(\theta',\phi_1)$. 
If $\ubd F > \overline{\dim}_{\mathrm A}^{\phi_1} P$ then clearly $\phi_1 = 1$ (so in fact $\overline{\dim}_{\mathrm A}^{\phi_1} P = \qad P$ by definition), and $f(\theta',\phi_1) = \ubd F = f(\theta,1) = \max_{\phi \in [\theta,1]} f(\theta,\phi)$. 
Therefore we may henceforth assume that $\ubd F \leq \overline{\dim}_{\mathrm A}^{\phi_1} P$. 
Using the quotient rule, if $\theta' < \phi_1$ then for all $\theta_1 \in (\theta',\min\{\phi_1,\theta\})$ we have $\frac{\partial f}{\partial \theta}(\theta_1,\phi_1) \geq 0$. Thus  
\[ f(\theta',\phi') \leq f(\theta',\phi_1) \leq f(\min\{\phi_1,\theta\},\phi_1). \]
Therefore if $\phi_1 \leq \theta$, then 
\[ f(\theta',\phi') \leq f(\phi_1,\phi_1) = \overline{\dim}_{\mathrm A}^{\phi_1} F \leq  \overline{\dim}_{\mathrm A}^{\theta} F = f(\theta,\theta) \leq \max_{\phi \in [\theta,1]} f(\theta,\phi).\]
If, on the other hand, $\phi_1 > \theta$, then 
\[ f(\theta',\phi') \leq f(\theta,\phi_1) \leq \max_{\phi \in [\theta,1]} f(\theta,\phi). \]
In either case, the required bound holds. 
\end{proof}

 Theorem~\ref{t:mainasp} provides bounds for the Assouad spectrum of a CIFS, and is the main result of this section. These bounds will be illustrated by examples and figures in Section~\ref{s:sharp}, which also show that the bounds are sharp. 
 Here and in subsequent sections, we assume only the conditions in Definition~\ref{d:cifs}, and not the additional separation condition from Section~\ref{s:asd}. 
Recall the definition of $f(\theta,\phi)$ from~\eqref{e:fdefn}, and recall that $h \coloneqq \dim_{\mathrm H} F$. 
\begin{thm}\label{t:mainasp}
Let $F$ be the limit set of any CIFS with notation as above. Let $P$ be any subset of $\cup_i S_i(X)$ which intersects $S_i(X)$ in exactly one point for each $i$. Then for all $\theta \in (0,1)$,  
\[ \max\{h,\uasp P \} \leq \uasp F = \asp F \leq \max_{\phi \in [\theta,1]} f(\theta,\phi). \]
\end{thm}

\begin{proof}
Fix any $\theta \in (0,1)$. 
By Lemma~\ref{l:samewithinlevelasd} we may assume without loss of generality that $P = \{ \, S_i(x) : i \in I \, \}$ for some $x \in \mathrm{Int}_{\Rd}(X) \cap F$. %
Since $\uasp$ is monotonic for subsets, $\uasp P \leq \uasp F$, and by~\cite[Theorem~3.15]{Mauldin1996}, $h = \dim_\mathrm{H} F \leq \uasp F$ so $\max\{h,\uasp P \} \leq \uasp F$. 
Lemma~\ref{l:monotoniclemma} says that the Assouad spectrum of $F$ is monotonic, so since $\overline{\dim}_\mathrm{A}^\theta F = \sup_{\theta' \in (0,\theta]} \dim_\mathrm{A}^{\theta'} F$, we have $\uasp F = \asp F$. 

It remains to show that $\uasp F \leq \max_{\phi \in [\theta,1]} f(\theta,\phi)$. 
 Let $t>\max_{\phi \in [\theta,1]} f(\theta,\phi)$ and fix $s \in (\max_{\phi \in [\theta,1]} f(\theta,\phi),t)$. 
 Let 
 \begin{equation}\label{e:defineeps}
 \epsilon \coloneqq \frac{1}{3}\Big(s - \max_{\phi \in [\theta,1]} f(\theta,\phi)\Big). 
 \end{equation}
 Let $C_{\epsilon} > 0$ be large enough that for all $r \in (0,1]$, $N_r(F) \leq C_{\epsilon} r^{-(\ubd F + \epsilon)}$. 
 We use the same constants $A_{d,r}, D, D', M$ from~\eqref{e:doublingconst},~\eqref{e:diameterslemma},~\eqref{e:2nddiams}, Lemma~\ref{l:assouadgeo},~\eqref{e:mdefn}. 
We introduce the constants: 
\begin{align*}
C_{\mathrm{small}} &\coloneqq A_{d,3} 2^{d + \epsilon}, \\
C_{\mathrm{sum}} &\coloneqq M 2^{2d + 3\epsilon/2} D' C_{\epsilon} \sum_{k=0}^{\infty} 2^{-(k+1)(\epsilon/2)} , \\
C_{\mathrm{med}} &\coloneqq M 2^{d + \epsilon} D' C_{\epsilon}, \\
C_{\mathrm{big}} &\coloneqq M D' D^s, \\
 C_{\mathrm{tot}} &\coloneqq 4 \max\{C_{\mathrm{small}}, C_{\mathrm{sum}}, C_{\mathrm{med}}, C_{\mathrm{big}}\}.
\end{align*}
Note that $C_{\mathrm{tot}} > C_{\mathrm{small}} > 1$. Fix $\lambda \in (0,1)$ small enough that 
\begin{equation}\label{e:lambdadef}
\lambda^{(1-\theta^{-1})(s-t)} < C_{\mathrm{tot}}^{-1}.
\end{equation} %
We now want to make all the contraction ratios small enough so that the constants in the induction argument below do not grow too fast. Suppose $||S_{i_1}'|| = \max_{i \in I}||S_i'|| \geq \lambda$. Then we can form the new CIFS $\{ \, S_i : i \in I \setminus \{i_1\} \, \} \cup \{ \, S_{i_1 j} : j \in I \, \}$. In a similar way, we can replace a map in this new CIFS whose derivative norm is largest. By Lemma~\ref{l:finitelymanylarge}, after repeating this finitely many times we will obtain a set $\{ \, T_j : j \in J \, \}$ with each $||T_j'|| < \lambda$. By induction, $\{T_j\}$ will form a CIFS with the same limit set as $\{S_i\}$, namely $F$. Moreover,~\eqref{e:diameterslemma},~\eqref{e:2nddiams}, Lemma~\ref{l:assouadgeo} and~\cite[Lemma~2.11]{Banaji2021} will hold with the \emph{same constants} $D, D', M$. Recalling that $P = \{ \, S_i(x) : i \in I \, \}$, define $Q \coloneqq \{ \, T_j(x) : j \in J \, \}$. Then $Q$ is the union of finitely many bi-Lipschitz copies of cofinite subsets of $P$. Thus if $\dim$ is any notion of dimension that is stable under bi-Lipschitz maps (for example any of the dimensions considered in this paper) then $\dim \{ \, S_i(x) : i \in I \, \} = \dim \{ \, T_j(x) : j \in J \, \}$.

The main step in the proof is the following claim. Inspired by Mauldin and Urba\'nski's proof of~\cite[Lemmas~2.8 and~2.10]{Mauldin1999} for the upper box dimension, this is proved using an inductive argument to construct efficient covers at smaller scales using efficient covers at larger scales. 

\begin{claim}\label{claim:asp}
There exists a large constant $A \in [1,\infty)$ such that for all $n \in \N$, if $x \in X$, $\lambda^n \leq R \leq 1$ and $0 < r \leq R^{1/\theta}$, then 
\[N_r(F \cap B(x,R)) \leq A C_{\mathrm{tot}}^n (R/r)^s.\] 
\end{claim}%

\begin{proof}[Proof of claim]
First note that $s > f(\theta,1) + \epsilon = \ubd F + \epsilon$. Therefore we can choose $\beta \in (\theta,1)$ to be close enough to 1 that 
\begin{equation}\label{e:definebeta}
 \left( \frac{1}{\beta} - 1\right)(d+\epsilon) \leq \left( \frac{1}{\theta} - 1\right)(s-\ubd F - \epsilon). 
 \end{equation}
Let $N \in \N$ be large enough that $\lambda^{-(N-1)} < R_{Q,\beta,\epsilon/2}/2$ and $\lambda^{(N-1)/\theta} < \lambda^{(N-1)/\beta}/10$, %
where $R_{Q,\beta,\epsilon/2}$ is the constant from Lemma~\ref{l:uniformconst}. 
Since $s > f(\theta,1) = \ubd F$, by choosing $A$ large enough we may assume the claim holds for $n=1,2,\ldots,N$. 
Suppose $n > N$ and assume the claim holds for $1,\ldots,n-1$. We may assume that $\lambda^n \leq R < \lambda^{n-1}$. 
Suppose $x \in X$ and $0 < r \leq R^{1/\theta}$. Define $\theta_r \in (0,\theta]$ and $k_r \in \N$ by $R=r^{\theta_r}$ and $R^{1/\beta} 2^{-(k_r + 2)} < r \leq R^{1/\beta} 2^{-(k_r + 1)}$. We break up the set of cylinders which intersect $B(x,R)$ into pieces depending on size. We will then bound the number of balls of radius $r$ needed to cover each piece separately. The following sets depend on $x$ and $n$: 
\begin{align*} 
\mathit{SMALL} &\coloneqq \{\, j \in J : F_j \cap B(x,R) \neq \varnothing, \, |X_j| \leq R^{1/\beta} 2^{-(k_r + 1)} \, \}, \\
I_k &\coloneqq \{ \, j \in J : F_j \cap B(x,R) \neq \varnothing, \, R^{1/\beta} 2^{-(k + 1)} < |X_j| \leq R^{1/\beta} 2^{-k} \, \} &\text{for } 0 \leq k \leq k_r, \\
\mathit{MED} &\coloneqq \{\, j \in J : F_j \cap B(x,R) \neq \varnothing, \,  R^{1/\beta} < |X_j| < R \, \}, \\
\mathit{BIG} &\coloneqq \{\, j \in J : F_j \cap B(x,R) \neq \varnothing, \, |X_j| \geq R \, \}, 
\end{align*}
Note that these sets are different to the sets $\mathit{SMALL}$ and $I_k$ from the proof of Theorem~\ref{t:asd}; those sets described indices whose cylinder sets intersect a different ball (not $B(x,R)$).

First consider $\mathit{SMALL}$. The cost of covering elements of $\mathit{SMALL}$ at scale $r$ is comparable to the cost of covering the corresponding elements of $Q$, so we will use the upper Assouad spectrum of $Q$ to obtain a bound. Since $2R \leq R_{Q,\beta,\epsilon/2}$, there exist at most $(2R/r)^{\uasp Q + \epsilon}$ balls of radius $r$ which cover $Q \cap B(x,2R)$. 
If $j \in \mathit{SMALL}$ then $Q \cap X_j \in B(x,2R)$, so the set of balls with the same centres and radii $3r$ will cover $\{ \, X_j : j \in \mathit{SMALL} \, \}$. By covering each of these larger balls with balls of radius $r$, 
\begin{equation}\label{e:smallcost}
 N_r\left( \bigcup_{j \in \mathit{SMALL}} X_j \right) \leq A_{d,3} 2^{d + \epsilon} (R/r)^{\uasp Q + \epsilon} \leq C_{\mathrm{small}} (R/r)^s.
 \end{equation}
The last inequality holds since $s > f(\theta,\theta) + \epsilon = \uasp Q + \epsilon$.

 Now consider $I_k$. We will bound the number of such cylinders using the upper Assouad spectrum of $Q$, and bound the cost of each using the upper box dimension of $F$. This is where the form of the function $f(\theta,\phi)$ comes from. Let $\phi_k \in [\theta_r,\beta)$ be such that $R^{1/\beta} 2^{-(k+1)} = R^{1/\phi_k}$. 
 For $0 \leq k \leq k_r$, 
 \begin{equation}\label{e:mediumcardinality}
 \begin{aligned}
  \# I_k &\leq \#  \{ \, j \in J : X_j \subset B(x,2R), \,  R^{1/\phi_k} \leq |X_j|  \leq R \, \} \\
 &\leq M N_{R^{1/\phi_k}} (Q \cap B(x,2R)) &&\text{($M$ is from \eqref{e:mdefn})} \\
 &\leq M \left( \frac{2R}{R^{1/\phi_k}}\right)^{\overline{\dim}_{\mathrm A}^{\phi_k} Q + \epsilon/2} &&\text{by Lemma~\ref{l:uniformconst}}\\
 &\leq M 2^{d + \epsilon/2} R^{(1-\phi_k^{-1})(\overline{\dim}_{\mathrm A}^{\phi_k} Q + \epsilon)} 2^{-(k+1)(\epsilon/2)} ,
 \end{aligned}
 \end{equation}
 where in the last line we use that $R^{1-\beta^{-1}} \geq 1$. 
  If $j \in I_k$ then by Lemma~\ref{l:assouadgeo}, 
\[ N_r(F_j) \leq D' N_{r/|X_j|} (F) \leq D' C_{\epsilon} (r/|X_j|)^{-(\ubd F + \epsilon)}
 \leq D' C_{\epsilon} 2^{d + \epsilon} R^{(\phi_k^{-1} - \theta_r^{-1})(\ubd F + \epsilon)}. \]
Therefore  
\begin{align}\label{e:mediumcost}
 \begin{split}
 N_r\left( \bigcup_{j \in \cup_{k=0}^{k_r} I_k} X_j \right) &\leq \sum_{k=0}^{k_r} N_r \left( \bigcup_{j \in I_k} X_j \right) \\
 &\leq C_{\mathrm{sum}} R^{(1-\phi_k^{-1})(\overline{\dim}_{\mathrm A}^{\phi_k} Q + \epsilon) + (\phi_k^{-1} - \theta_r^{-1})(\ubd F + \epsilon)} \\
 &\leq C_{\mathrm{sum}} (R/r)^s.
 \end{split}
 \end{align}
 The last inequality is since $s > \max_{\phi \in [\theta,1]} f(\theta,\phi) + 2\epsilon \geq f(\theta_r,\phi_k) + 2\epsilon$ by the definition of $\epsilon$ in~\eqref{e:defineeps} and Lemma~\ref{l:fbound}.

 Now consider $\mathit{MED}$. We will use that $\beta$ is close to 1 to show that the cardinality of $\mathit{MED}$ is not too large. We then use the upper box dimension of $F$ to bound the cost of each element of $\mathit{MED}$. As in~\eqref{e:mediumcardinality}, 
 \[ \# \mathit{MED} \leq M \left( \frac{2R}{R^{1/\beta} }\right)^{\overline{\dim}_{\mathrm A}^{\beta} Q + \epsilon} \leq M 2^{d + \epsilon} R^{(1-\beta^{-1})(d + \epsilon)}. \]
 If $j \in \mathit{MED}$ then 
 \[N_r(F_j) \leq D' N_{r/R}(F) \leq D' C_{\epsilon} (r/R)^{-(\ubd F + \epsilon)}. \] 
 Using the definition of $\beta$ in~\eqref{e:definebeta}, and since $r \leq R^{1/\theta}$, 
 \begin{equation}\label{e:medcost}
 N_r\left( \bigcup_{j \in \mathit{MED}} X_j \right) \leq M 2^{d + \epsilon} R^{(1-\beta^{-1})(d + \epsilon)} D' C_{\epsilon} (r/R)^{-(\ubd F + \epsilon)} \leq C_{\mathrm{med}} (R/r)^s. 
 \end{equation}

 Finally, consider $\mathit{BIG}$. The conformality and OSC give an absolute bound for the cardinality: $\# \mathit{BIG} \leq M$ from~\eqref{e:mdefn}. We now use conformality (through Lemma~\ref{l:assouadgeo}) to compare the cost of each piece with the cost of its (larger) preimage, which can be bounded using the inductive hypothesis. Indeed, if $j \in \mathit{BIG}$ then 
\begin{align*}
 N_r( B(x,R) \cap F_j) &\leq D' N_{r/||T_{j}'||}(T_{j}^{-1}(B(x,R) \cap F_j)) &\text{by Lemma~\ref{l:assouadgeo}}\\
 &\leq D' A C_{\mathrm{tot}}^{n-1} \left( \frac{DR/||T_{j}'||}{r/||T_{j}'||} \right)^s &\text{by inductive hypothesis} \\
 &= A C_{\mathrm{tot}}^{n-1} D' D^s (R/r)^s.
 \end{align*}
 We were able to apply the inductive hypothesis at the crucial step because $r \leq R^{1/\theta}$ and $||T_{j}'|| \leq \lambda \leq 1$ so 
 \[ r/||T_{j}'|| \leq (R/||T_{j}'||)^{1/\theta} \leq (RD/||T_{j}'||)^{1/\theta} \]
 and 
 \[ DR/||T_{j}'|| \geq R/\lambda \geq \lambda^{n-1}.\]
 Now, 
 \begin{equation}\label{e:bigcost} N_r\left( \bigcup_{j \in \mathit{BIG} } X_j \right) \leq A C_{\mathrm{big}} C_{\mathrm{tot}}^{n-1} (R/r)^s. 
 \end{equation}
 
 Putting together~\eqref{e:smallcost},~\eqref{e:mediumcost},~\eqref{e:medcost},~\eqref{e:bigcost} and using the definition of $C_{\mathrm{tot}}$ gives  
  \[ N_r(F \cap B(x,R)) \leq (C_{\mathrm{small}} + C_{\mathrm{sum}} + C_{\mathrm{med}} + A C_{\mathrm{big}} C_{\mathrm{tot}}^{n-1})(R/r)^s \leq A C_{\mathrm{tot}}^n (R/r)^s, \]
  completing the proof of the claim. 
\end{proof}

We now complete the proof of Theorem~\ref{t:mainasp}. If $n \in \N$, $\lambda^n \leq R \leq 1$, $x \in X$ and $0 < r \leq R^{1/\theta}$, then 
\begin{align}\label{e:exponentialgap}
N_r(F \cap B(x,R)) \leq A C_{\mathrm{tot}}^n (R/r)^s &\leq A C_{\mathrm{tot}}^n R^{(1-\theta^{-1})(s-t)} (R/r)^t \\ \nonumber
&\leq  A C_{\mathrm{tot}}^n \lambda^{n (1-\theta^{-1})(s-t)}(R/r)^t \\ \nonumber
&\leq A(R/r)^t. \nonumber
\end{align}
The last inequality is by the definition of $\lambda$ in~\eqref{e:lambdadef}. Thus $\uasp F \leq t$, as required.  
\end{proof}

In~\eqref{e:exponentialgap}, we exploited the exponential gap between the scales $r \leq R^{1/\theta}$ and $R$ that is in the definition of the Assouad spectrum but not the Assouad dimension. This allowed us to complete the proof of Theorem~\ref{t:mainasp} without assuming the separation condition that is assumed in Theorem~\ref{t:asd}.

\subsection{Consequences of Theorem \ref{t:mainasp}}\label{s:consequences}

We now consider what can be deduced from Theorem~\ref{t:mainasp} about the general form of the Assouad spectrum of infinitely generated self-conformal sets. Throughout Section~\ref{s:consequences}, $F$ will denote the limit set of a CIFS, $P$ will denote an arbitrary subset of $\cup_i S_i(X)$ intersecting $S_i(X)$ in exactly one point for each $i$, and $h = \dim_{\mathrm H} F$. 
Recall that the quasi-Assouad dimension is defined in~\eqref{e:quasiassouaddef}. 
\begin{cor}\label{c:qa}
We have $\qad F = \max\{h,\qad P\}$.
\end{cor}
\begin{proof}
Since $\ubd F = \max\{h,\ubd P\}$, we have $\max_{\phi \in [\theta,1]} f(\theta,\phi) \leq \max\{h,\qad P\}$ for all $\theta \in (0,1)$, so the result follows from Theorem~\ref{t:mainasp} upon letting $\theta \to 1^-$. 
\end{proof}

\begin{cor}%
If $h \geq \qad P$ then $\asp F = h$ for all $\theta \in (0,1)$. 
\end{cor}
\begin{proof}
Immediate from Corollary~\ref{c:qa}. 
\end{proof}

\begin{cor}
If $\asp P = \ubd P$ for all $\theta \in (0,1)$ then $\asp F = \ubd F$ for all $\theta \in (0,1)$. 
\end{cor}
\begin{proof}
This follows from Corollary~\ref{c:qa} and the fact that $\ubd F = \max\{h,\ubd P\}$. 
\end{proof}

Recall that the phase transition $\rho_F$ for the Assouad spectrum of a set $F$ is defined in~\eqref{e:definephasetransition}. 
\begin{cor}\label{c:phase}
If $h < \qad P$ then $\rho_F = \rho_P$. 
\end{cor}
\begin{proof}
It follows from Lemma~\ref{l:samewithinlevelasd} and Corollary~\ref{c:qa} that $\rho_F \leq \rho_P$, so it remains to prove the reverse inequality. Fix $\theta \in (0,\rho_P)$, so $f(\theta,\theta) = \uasp P < \qad P$. 
If $\phi \in (\theta,1]$ then 
\[ f(\theta,\phi) \leq \frac{(\phi^{-1} - 1) \qad P + (\theta^{-1} - \phi^{-1}) \max\{h,\uasp P\}}{\theta^{-1} - 1}   < \qad P.\] 
Therefore $\uasp F < \qad P$ by Theorem~\ref{t:mainasp}, so $\rho_F \geq \theta$. It follows that $\rho_F \geq \rho_P$, as required. 
\end{proof}
In light of Corollary~\ref{c:phase}, when there is no confusion we will sometimes write $\rho$ for the common value $\rho_P = \rho_F$.

\begin{defn}\label{d:threeparam}
If $G \subset \Rd$ is non-empty and bounded then we say that the Assouad spectrum of $G$ has the \emph{three-parameter form} if either $\ubd G = \qad G$, or else 
\[
 \asp G = \min\left\{ \ubd G + \frac{(1-\rho_G) \theta}{(1-\theta)\rho_G} (\dim_\mathrm{qA} G - \ubd G) , \dim_\mathrm{qA} G \right\}
 \]
 for all $\theta \in [0,1)$. 
 \end{defn}
 The three parameters are the upper box dimension, the quasi-Assouad dimension, and the phase transition $\rho_G \in \left[1-\frac{\ubd G}{\dim_\mathrm{qA} G},1\right]$ (defined in~\eqref{e:definephasetransition}). 
The Assouad spectrum of many natural sets, such as polynomial sequences and spirals, Bedford--McMullen carpets, Kleinian limit sets and Julia sets happens to take the three-parameter form (see~\cite[Section~17.7]{Fraser2020book} and~\cite{FraserPreprintsullivan}). 
It is perhaps noteworthy that infinitely generated self-conformal sets do not necessarily satisfy the three-parameter form (as we will see in Theorem~\ref{t:sharp}). However, they sometimes will satisfy this form, as the following result shows.

\begin{cor}\label{c:special}
Assume that the Assouad spectrum of $P$ is non-constant and has the three-parameter form as in Definition~\ref{d:threeparam}. 
Then the upper bound of Theorem~\ref{t:mainasp} is the three-parameter form for $F$, namely  
\begin{equation}\label{e:specialubf}
 \asp F \leq \max_{\phi \in [\theta,1]} f(\theta,\phi) = \min\left\{ \ubd F + \frac{(1-\rho_F) \theta}{(1-\theta)\rho_F} (\dim_\mathrm{qA} F - \ubd F) , \dim_\mathrm{qA} F \right\}. 
 \end{equation}
In particular, if $h \leq \ubd P$ then the bounds in Theorem~\ref{t:mainasp} coincide, with $\asp F = \asp P$ for all $\theta \in [0,1]$. 
If, on the other hand, we have $\ubd P < h < \qad P$, then the upper and lower bounds of Theorem~\ref{t:mainasp} differ for all $\theta \in (0,\rho_P)$. 
\end{cor}

\begin{proof}
\emph{Case 1:} If $h \geq \qad P$ then clearly $\max_{\phi \in [\theta,1]} f(\theta,\phi) = f(\theta,1) = h$. %

\emph{Case 2:} If $h \leq \ubd P = \qad P$ then $\max_{\phi \in [\theta,1]} f(\theta,\phi) = f(\theta,\theta) = \ubd P$. 

\emph{Case 3:} If $h \leq \ubd P < \qad P$ then $\max_{\phi \in [\theta,1]} f(\theta,\phi) = f(\theta,\phi') = \uasp P$ for all $\phi' \in [\theta,\rho_P]$, and the bounds coincide. 

\emph{Case 4:} If $\ubd P < h < \qad P$ then a direct calculation shows that for $\theta \in (0,\rho_P)$, 
\[ \max_{\phi \in [\theta,1]} f(\theta,\phi) = f(\theta,\rho_P) = h + \frac{(1-\rho_P) \theta}{(1-\theta)\rho_P}(\qad P - h) > \asp P.\]%
We see that in all cases $\asp F$ has the three-parameter form. 
\end{proof}

One setting where Corollary~\ref{c:special} is applicable is where the maps are concentrated around the sets $F_p \coloneqq \{ \, i^{-p} : i \in \N \, \}$, where $p \in (0,\infty)$ is a constant. Fraser and Yu~\cite[Corollary~6.4]{Fraser2018-2} showed that 
\begin{equation}\label{e:fpspectrum}
 \asp F_p = \min\left\{ \frac{1}{(1+p)(1-\theta)},1\right\}.
 \end{equation}
This is the three-parameter form with $\ubd F_p = (p+1)^{-1}$, $\qad F_p = 1$, $\rho_{F_p} = \frac{p}{p+1}$. 
\begin{lma}\label{l:existcifstechnical}
Let $p \in (0,\infty)$, $t \in [p+1,\infty)$ and $h \in (1/t,1)$. Then there exists a CIFS on $X=[0,1]$ with $I=\{2,3,4,\ldots\}$ %
and limit set $F$ such that the following hold: 
\begin{itemize}
\item For all $i \in \N$, $i \geq 2$ there exists $c_i \in (0,(i-1)^{-p} - i^{-p})$ such that $S_i(x) = c_i x + i^{-p}$ for all $x \in [0,1]$, so $S_i$ is a similarity map with contraction ratio $c_i$ and $i^{-p} = S_i(0) < S_i(1) \leq (i-1)^{-p}$. 
\item There exists $N \in \N$ such that $c_i = p i^{-t}$ for all $i \geq N$. %
\item $\dim_\mathrm{H} F = h$
\end{itemize} 
\end{lma}
In fact any CIFS which satisfies the first two conditions will have finiteness parameter $\Theta = 1/t$ and $\dim_\mathrm{H} F \in (1/t,1]$. 

\begin{proof}
By a mean value theorem argument, $(i-1)^{-p} - i^{-p} > p i^{-(p+1)}$ for all $i \in \N$, $i \geq 2$. By~\eqref{e:hausdorffpressure} we have $h = \inf\{\, s \geq 0 : \sum_{i=2}^\infty c_i^s < 1 \, \}$. Therefore since $1/t < h < 1$, we can choose $N \geq 2$ sufficiently large that $p^h \sum_{i=N+1}^\infty  i^{-th} < 1$ %
and $\sum_{i=1}^N ((i-1)^{-p} - i^{-p})^h \geq 1$. %
By an intermediate value theorem argument, there exist $c_j \in (0,(j-1)^{-p} - j^{-p})$ for $1 \leq j \leq N$ such that 
\[ \sum_{j=1}^N c_j^h + p^h \sum_{i=N+1}^\infty  i^{-th} = 1.\]  
It now follows from~\cite[Theorem~3.15]{Mauldin1996} that $\dim_{\mathrm H} F = h$. 
\end{proof}

We now use Lemma~\ref{l:existcifstechnical} to give an example with $\ubd F = \ubd F_p$ and the Assouad spectrum satisfying the three-parameter form. 
\begin{cor}
Consider a CIFS satisfying the three conditions in Lemma~\ref{l:existcifstechnical} with $t>p+1$ and $h = \dim_\mathrm{H} F \in (1/t,(p+1)^{-1}]$. Then for all $\theta \in [0,1]$,  
\[\asp F = \asp F_p = \min\left\{ \frac{1}{(1+p)(1-\theta)},1\right\}. \]%
\end{cor}

\begin{proof}
This is immediate from Corollary~\ref{c:special}. 
\end{proof}

\section{Sharpness of the bounds and attainable forms of Assouad spectra}\label{s:sharp}

Theorem~\ref{t:sharp} provides a family of examples with $\dim_{\mathrm B} F = \dim_\mathrm{H} F \eqqcolon h$ which show in particular that the bounds in Theorem~\ref{t:mainasp} are sharp in general. %
The graph of the Assouad spectrum for a certain choice of parameters is shown in Figure~\ref{f:sharp}.

\begin{thm}\label{t:sharp}
Consider a CIFS satisfying the three conditions in Lemma~\ref{l:existcifstechnical} with $h \in ((p+1)^{-1},1)$. There are three different cases depending on the parameter $t$: 

\begin{enumerate}
\item\label{i:sharpupper} If $t=p+1$ then 
\begin{equation}\label{e:sharpupperspecial}
  \nonumber \asp F = \left\{\begin{array}{lr}
        h + \frac{\theta}{p(1-\theta)}(1-h), & \text{for } 0\leq \theta < \frac{p}{1+p} \\
        1, & \text{for } \frac{p}{1+p}\leq \theta \leq 1
        \end{array}\right. . 
\end{equation}%

\item\label{i:sharpmiddle} If $p+1 < t < p+h^{-1}$ then 
\[ \asp F = 
\left\{\begin{array}{lr}
        h + \frac{\theta}{p(1-\theta)}(1-h(t-p)), & \text{for } 0\leq \theta \leq \frac{(h+hp-1)p}{(1+p)(ht-1)}\\
        \frac{1}{(1+p)(1-\theta)}, & \text{for } \frac{(h+hp-1)p}{(1+p)(ht-1)} < \theta < \frac{p}{1+p}\\
        1, & \text{for } \frac{p}{1+p}\leq \theta \leq 1
        \end{array}\right. .\]

\item\label{i:sharplower} If $t \geq p + h^{-1}$ then 
\[ \asp F = 
    \left\{\begin{array}{lr}
        h, & \text{for } 0\leq \theta \leq \frac{h+hp-1}{h(1+p)}\\
        \frac{1}{(1+p)(1-\theta)}, & \text{for } \frac{h+hp-1}{h(1+p)} < \theta < \frac{p}{1+p}\\
        1, & \text{for } \frac{p}{1+p}\leq \theta \leq 1
        \end{array}\right. . \]
\end{enumerate}
\end{thm}

By Corollary~\ref{c:special}, the bounds from Theorem~\ref{t:mainasp} (or from Corollary~\eqref{c:special}) differ if $0<\theta< \rho_{F_p} = \frac{p}{1+p}$. 
Moreover, in~\eqref{i:sharpupper} the upper bounds are attained, in~\eqref{i:sharpmiddle} $\asp F$ lies strictly between the bounds for all $\theta \in (0,\frac{(h+hp-1)p}{(1+p)(ht-1)})$, and in~\eqref{i:sharplower} the lower bounds are attained. We note that the formula in~\eqref{i:sharplower} does not depend on the precise value of $t \in [p+h^{-1},\infty)$.

\begin{figure}[ht]
\center{\includegraphics[width=0.6\textwidth]
        {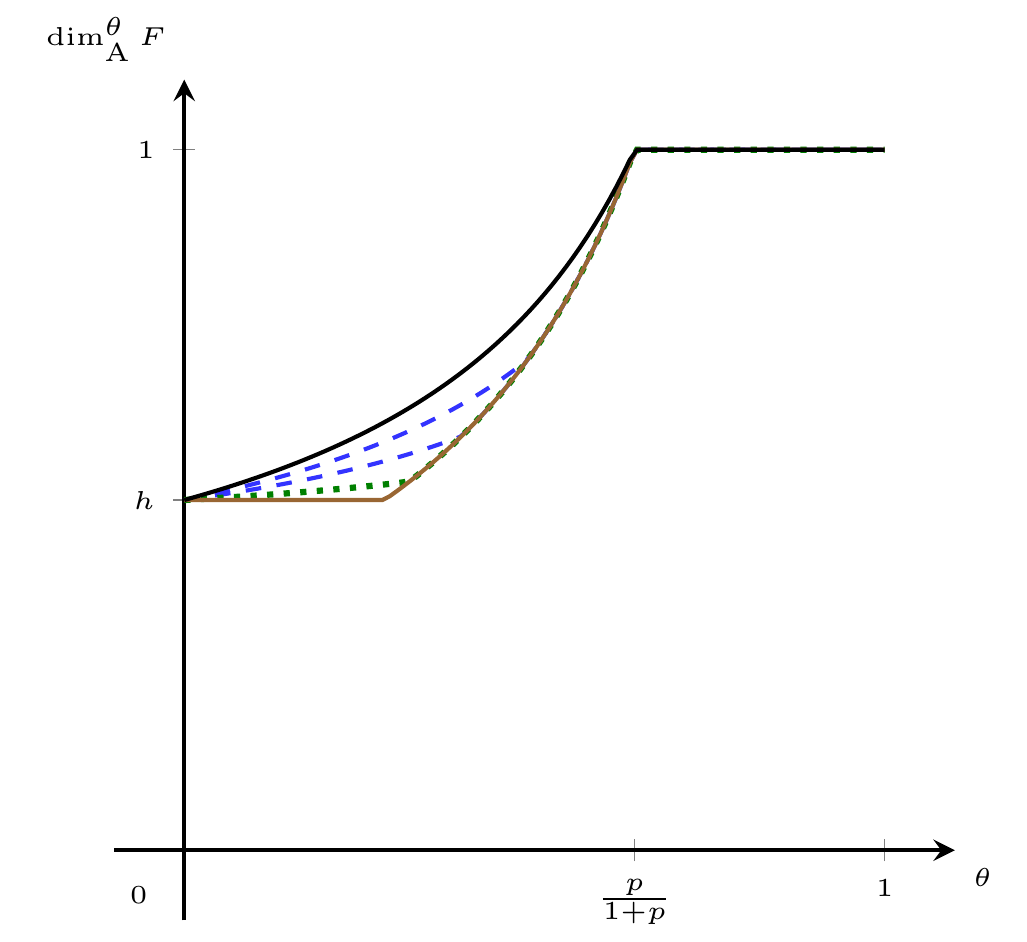}}
        \caption{\label{f:sharp}
        Assouad spectra of the sets in Theorem~\ref{t:sharp} for $p=1.8$ and $h \approx 0.5$. In black: the upper bound (attained when $t=p+1$). In \textcolor{brown!80!black}{brown}: the lower bound (attained when $t\geq p+h^{-1}$). In dashed \textcolor{blue!80}{blue:} some more choices of $t \in (p+1,p+h^{-1})$. In dotted \textcolor{green!50!black}{green:} the case $t=2p$, which is also the dimension of the continued fraction set from Proposition~\ref{p:ctdspaced}. %
 }
 
\end{figure}

We now give a technical lemma that forms the main part of the proof of the upper bound of Theorem~\ref{t:sharp}~\eqref{i:sharpmiddle}. 
\begin{lma}\label{l:sharplemma}
Consider a CIFS satisfying the conditions of Lemma~\ref{l:existcifstechnical} with $(p+1)^{-1} < h < 1$ and $p+1 < t < p + h^{-1}$. 
Since for each $i \geq 2$, $S_i([0,1])$ and $S_{i+1}([0,1])$ are sufficiently well separated, we can choose $c>0$ large enough that for any interval $B$ (of length $R$, say) there exists at most one $i \in I$ such that $|F_i| \geq c R$ and $F_i \cap B \neq \varnothing$. 
Let $s$ be any number greater than the claimed formula for $\asp F$ in Theorem~\ref{t:sharp}~\eqref{i:sharpmiddle}. Then if $0 < r < 1$ and $B$ is any interval of length $r^\theta$ such that whenever $i \in I$ is such that $B \cap F_i \neq \varnothing$ we have $|F_i| \leq cr^\theta$, then 
\[ N_r(B \cap F) \lesssim r^{(\theta - 1)s}, \]
where the implicit constant in $\lesssim$ can depend on $p,h,t, F,s$ but not on $r$ or $B$. 
\end{lma}
\begin{proof}
By increasing the implicit constant if required, we may assume that $r$ is small enough that if $B \cap S_i([0,1]) \neq \varnothing$ then $S_i([0,1]) = [i^{-p}, i^{-p} + p i^{-t}]$. 
Since $B \cap F \subseteq \bigcup_{i \in I : B \cap F_i \neq \varnothing} F_i$, it suffices to prove that $N_r(\bigcup_{i \in I : B \cap F_i \neq \varnothing} F_i) \lesssim r^{(\theta-1)s}$. 
First note that since $s > \asp P$, 
\[N_r \left(\bigcup_{i \in I : B \cap F_i \neq \varnothing, |F_i| \leq r} F_i \right) \lesssim N_r(B \cap P) \lesssim r^{(\theta - 1)s}.\]
Therefore it remains to prove that 
$N_r(\bigcup_{i \in I : B \cap F_i \neq \varnothing, r < |F_i| \leq cr^\theta} F_i ) \lesssim r^{(\theta - 1)s}$. To do so, we consider three cases depending on the value of $\theta$, corresponding to the location of the point $i^{-p}$ that satisfies $i^{-t} \approx r$. For such an $i$, in Case 1, $i^{-p} \lesssim r^\theta$, in Case 2, $r^\theta \lesssim i^{-p} \lesssim r^{\frac{p\theta}{p+1}}$, and in Case 3, $r^{\frac{p\theta}{p+1}} \lesssim i^{-p}$. The significance of $r^{\frac{p\theta}{p+1}}$ is that this is where gaps between consecutive elements of $P$ are $\approx r^\theta$. %

\textbf{Case 1:} Assume $0 < \theta \leq p/t$. %
Since $s > h = \dim_{\mathrm B} F$, we have 
\begin{align*} 
N_r\left( \bigcup_{i \in I : B \cap F_i \neq \varnothing, r < |F_i| \leq r^{t\theta/p} } F_i \right) &\lesssim \sum_{i = \lfloor r^{-\theta/p} \rfloor}^{\lfloor r^{-1/t} \rfloor} N_{r i^t}(F) \\
&\approx \int_{r^{-\theta /p}}^{r^{-1/t}} (r x^t)^{-h} dx  \\
&\lesssim r^{-h} (r^{1-th})^{-\theta/p} \\
&\leq r^{(\theta - 1)s}. 
\end{align*}
Let $\eta$ be such that $\sup B = r^{\eta}$. 
We may assume without loss of generality that $\eta \leq \theta$. %
We now have two subcases depending on the value of $\eta$. 

\begin{itemize}
\item Subcase 1.1: Assume $\eta > \frac{\theta p}{p+1}$. 
If $th \geq p+1$ then $\frac{\eta(p+1-th) + ph - p\theta}{p(1-\theta)} < h$, %
whereas if $th < p+1$ then $\frac{\eta (p+1-th) + ph - p\theta}{p(1-\theta)} \leq h + \frac{\theta}{p(1-\theta)}(1-h(t-p))$, so if we define 
\[ \epsilon \coloneqq \frac{1-p/t}{2}\left(s - \left( h + \frac{\theta}{p(1-\theta)}(1-h(t-p)) \right) \right) > 0 \]
 then $\frac{\eta(p+1-th) + ph - p\theta}{p(1-\theta)} + \frac{2}{1-p/t}\epsilon \leq s$. 
 Then 
\begin{align*} 
\# \{ \, i \in I : B \cap F_i \neq \varnothing, |F_i| > r  \, \} &\lesssim  r^{-\eta /p} - (r^{\eta} + r^\theta)^{-1/p} \\
&= r^{-\eta /p}(1 - (1+r^{\theta - \eta})^{-1/p}) \\
&\lesssim r^{-\eta /p + \theta - \eta - \epsilon}, 
\end{align*}
where the last bound is by Taylor's theorem. 
Furthermore, if $F_w \cap B \neq \varnothing$ and $|F_i| > r$ then $|F_i| \approx r^{\eta t/p}$, and $N_r(F_i) \approx N_{r^{1-\eta t/p}}(F) \lesssim r^{-(1-\eta t/p)(h+\epsilon)}$. 
 Putting this all together gives 
\begin{align*} 
N_r \left( \bigcup_{i \in I : B \cap F_i \neq \varnothing, |F_i| > r} F_i \right) &\lesssim r^{-\eta /p + \theta - \eta - \epsilon - (1-\eta t/p)(h+\epsilon)} \\
&= r^{(\theta - 1)( \frac{\eta (p+1-th) + ph - p\theta}{p(1-\theta)}  + \frac{2}{1-p/t}\epsilon)} \\
&\leq r^{(\theta - 1)s}. 
\end{align*}

\item Subcase 1.2: Assume $\eta \leq \frac{\theta p}{p+1}$. Then since $|B| = r^\theta$, we have $\# \{ \, i \in I : B \cap F_i \neq \varnothing\} \lesssim 1$. Therefore 
\[ N_r \left( \bigcup_{i \in I : B \cap F_i} F_i \right) \lesssim N_{r/c r^\theta}(F) \lesssim r^{(\theta - 1)s}. \]
\end{itemize}

\textbf{Case 2:} Assume $p/t < \theta < (p+1)/t$. Again let $\eta$ be such that $\sup B = r^{\eta}$. If $\eta > \frac{\theta p}{p+1}$ then by a similar argument to the proof of Subcase 1.1 we have 
\[ N_r \left( \bigcup_{i \in I : B \cap F_i \neq \varnothing, |F_i| > r} F_i \right) \lesssim r^{(\theta - 1)s}.\]
 If, on the other hand, $\eta \leq \frac{\theta p}{p+1}$, then as in Subcase 1.2 we have $\# \{ \, i \in I : B \cap F_i \neq \varnothing\} \lesssim 1$ so $N_r (\bigcup_{i \in I : B \cap F_i \neq \varnothing} F_i ) \lesssim r^{(\theta - 1)s}$. 

\textbf{Case 3:} Assume $(p+1)/t \leq \theta < 1$. Then if $r < |F_i| \leq c r^\theta$ and $B \cap F_i \neq \varnothing$ then $r^{\frac{\theta p}{p+1}} \leq r^{p/t} \lesssim \inf B$. %
Therefore
\[ \# \{ \, i \in I : B \cap F_i \neq \varnothing, r < |F_i| \leq c r^\theta \, \} \lesssim 1. \]
 Consequently, 
\[ N_r \left( \bigcup_{i \in I : B \cap F_i \neq \varnothing, r < |F_i| \leq c r^\theta } F_i \right) \lesssim r^{(\theta - 1)s}. \qedhere \]
\end{proof}

\begin{proof}[Proof of Theorem~\ref{t:sharp}]
We begin with the upper bounds. 
Case~\eqref{i:sharpupper} follows from Theorem~\ref{t:mainasp}. 
Assume $p + 1 < t < p + h^{-1}$ and let $s$ be larger than the claimed upper bound. 
The argument now uses a similar trick to the proof of Theorem~\ref{t:asd}. Let $0 < r < 1$ and let $B$ be a ball of radius $r^\theta$ intersecting $F$. Recall the constant $c$ from Lemma~\ref{l:sharplemma}. Let $w \in I^*$ be the unique word such that $F_w \cap B \neq \varnothing$ and $|F_w| \geq cr^\theta$ and such that if $v \in I^*$ satisfies $F_v \cap B \neq \varnothing$ and $|F_v| \geq cr^\theta$ then $v$ is a subword of $w$. Then $B \cap F = B \cap F_w$. If $c_w$ is the contraction ratio of $S_w$ then 
\[ N_r(B \cap F) = N_{r/c_w}(S_w^{-1}(B \cap F)) \leq N_{r/c_w}(S_w^{-1}(B) \cap F) \lesssim r^{(\theta - 1)s} \]
by Lemma~\ref{l:sharplemma}, since the claimed formula is increasing in $\theta$. The upper bound now follows. Case~\ref{i:sharplower} can be proved using a similar (but easier) argument. 

We now consider the lower bounds. Case~\eqref{i:sharplower} follows from Corollary~\ref{c:special}, so assume $t < p+h^{-1}$. 
The bound for $\theta \geq \frac{(h+hp-1)p}{(1+p)(ht-1)}$ is immediate from Theorem~\ref{t:mainasp}. Since $p/t \geq \frac{(h+hp-1)p}{(1+p)(ht-1)}$, it suffices to prove 
\[ \asp F \geq h + \frac{\theta}{p(1-\theta)}(1-h(t-p)) \qquad \mbox{ for all }\theta \in (0,p/t). \]%
Fix $\epsilon \in (0,(th-1)/t)$. %
In the $\lesssim$ notation below, the implicit constant can depend on $p,h,t,\epsilon$ only. 
Since $\dim_{\mathrm B} F = h$ we have $N_r(S_i(F)) \approx N_{r i^t} (F) \gtrsim (r i^t)^{-(h-\epsilon)}$ %
for sufficiently small $r>0$ and all $i \in \N$ such that $r \leq i^{-t}$, by Lemma~\ref{l:assouadgeo}. Therefore
 \begin{align*} 
N_r(F \cap [0,r^\theta]) &\gtrsim \sum_{i = \lfloor r^{-\theta / p} \rfloor}^{\lfloor r^{-1/t}\rfloor} (ri^t)^{-(h-\epsilon)} \\
&\approx r^{-(h-\epsilon)} \int_{ r^{-\theta / p} }^{r^{-1/t}} x^{-t(h-\epsilon)} dx \\
&\approx r^{-(h-\epsilon)} r^{-\frac{\theta}{p} (1-t(h-\epsilon))} \\
&= r^{(\theta - 1)\left(  h + \frac{\theta}{p(1-\theta)}(1-h(t-p)) - \frac{1 - \theta t / p}{1 - \theta } \epsilon  \right) }. 
 \end{align*} 
 Since $\epsilon$ was arbitrary, this completes the proof. 
\end{proof}

We note that the Assouad spectrum can be used to distinguish between different sets in this setting in cases where no other notion of dimension is able to. 
For any fixed $p \in (0,\infty)$, consider two sets from Theorem~\ref{t:sharp} that are chosen to have the same Hausdorff dimension $h$ but use different values of $t \in [p+1,p+h^{-1}]$, so have different Assouad spectra. 
Since the Assouad spectrum is stable under bi-Lipschitz maps, we deduce that the two sets cannot be bi-Lipschitz equivalent. 
In fact,~\cite[Proposition~4.7]{Fraser2018-2} can be used to give quantitative information about the possible exponents of bi-H\"older maps between two such sets. %
However, the Hausdorff, box and intermediate dimensions (studied in~\cite{Banaji2021}) of each of the sets will be $h$, and the Assouad and quasi-Assouad dimensions will be 1, independent of $t$, so none of these other dimensions provide information about Lipschitz equivalence or H\"older distortion.

The Assouad spectra of the sets in Theorem~\ref{t:sharp}~\eqref{i:sharpupper} satisfy the three-parameter formula and so have only one phase transition. The sets in~\eqref{i:sharpmiddle} and~\eqref{i:sharplower}, on the other hand, are sets which exhibit self-similarity and whose Assouad spectrum displays interesting behaviour in that it 
\begin{itemize}
\item does \emph{not} satisfy the three-parameter formula,
\item has two phase transitions.
\end{itemize}
These are the first examples of dynamically defined fractals whose Assouad spectrum has two phase transitions, though the elliptical polynomial spirals in~\cite{Burrell2020-1} can also have Assouad spectrum with two phase transitions. 
We will see in Section~\ref{s:ctdfracsect} that this behaviour can also be observed for classes of infinitely generated self-conformal sets defined using continued fraction expansions. 

If we do not insist that the set of fixed points forms a polynomial sequence, then we will see in Proposition~\ref{p:arbitrary} by making the contraction ratios extremely small that the Assouad spectrum of the limit set can be very wild, and indeed display behaviour that is as general as what is possible for arbitrary sets. 
To prove this we need Lemma~\ref{l:countable}, whose proof uses homogeneous Moran sets. 
These are topological Cantor sets formed by starting with a closed interval, removing an open interval from its centre, removing an open interval from the centre of each of the remaining intervals, and repeating this process indefinitely. 
The size of each interval removed is constant at each level but can vary between levels, resulting in a set which is very homogeneous at each given scale, but which can look much sparser at some scales than others. 
Homogeneous Moran sets have proved very useful in the theory of dimension interpolation and are described more formally in~\cite{Banaji2022moran,Rutar2022assouad}. 

\begin{lma}\label{l:countable}
If $G \subset \R$ is non-empty and bounded then there exists a countable set $P \subset \R$ which accumulates only at $0$ such that $\asp P = \asp G$ for all $\theta \in (0,1)$. 
\end{lma} 
\begin{proof}
It is known that there exists a homogeneous Moran set $M$ (whose initial interval is $[0,1]$, say) with $\asp M = \asp G$ for all $\theta \in (0,1)$. Indeed, this was first proved in~\cite{Rutar2022assouad} using techniques from~\cite{Banaji2022moran}, and a more explicit proof was subsequently described in \cite[Remark~1.4]{BanajiPreprintphiassouad}. 
For each $n \geq 2$, let $P_n$ be a maximal $n^{-n}$-separated subset of $M \cap [0,1/n]$, and let $P \coloneqq \cup_{n = 1}^{\infty} P_n$, so $P$ is countable and accumulates only at $0$. 
Since $P \subset M$, $\asp P \leq \asp M = \asp G$ for all $\theta$. 

For the reverse inequality, we note that for all $\theta \in (0,1)$, for all $n$ sufficiently large (depending on $\theta$) and $r \in [n^{-1/\theta},n^{-1}]$, $N_r(P \cap [0,1/n]) \approx N_r(M \cap [0,1/n])$. 
But since $M$ is a \emph{homogeneous} Moran set, $N_r(M \cap [0,1/n]) \gtrsim N_r(M \cap [x,x+1/n])$ for all $x \in \R$. Therefore $\asp P \geq \asp M = \asp G$, completing the proof. 
\end{proof} 

\begin{prop}\label{p:arbitrary}
Let $G \subset \R$ be \emph{any} non-empty bounded set with $\qad G > 0$. Then there exists a CIFS of similarity maps on $[0,1]$ whose limit set $F$ satisfies $\asp F = \uasp G$ for all $\theta \in (0,1)$. 
\end{prop}
\begin{proof}
Let $P$ be the set from Lemma~\ref{l:countable} associated to $G$, and write $P = \{p_1,p_2,\dotsc \}$ where $p_1 > p_2 > \dotsb$. 
Choose positive numbers $\{c_1,c_2,\dotsc \}$ small enough that $\{ \, S_i(x) = c_i x + p_i : i \in \N \, \}_{i \in I}$ forms a CIFS on $[0,1]$. 
Since $\qad G > 0$, it is known from work in \cite{Fraser2019-3,Fraser2018-2,Garcia2021qa} that $\ubd G > 0$. 
By decreasing the $c_i$ if necessary, we may assume that the Hausdorff dimension $h$ of the limit set $F$ of our CIFS satisfies $h \leq \ubd G$. 
Decreasing the $c_i$ even further, we may assume that $c_{n+1} \leq c_n^n$ for all $n \in \N$. Then following the proof of Theorem~\ref{t:mainasp}, we see that for all $\theta \in (0,1)$ and $n$ sufficiently large (depending on $\theta$), the sets defined in Claim~\ref{claim:asp} satisfy 
\[ \# \left( \mathit{MED} \cup \bigcup_{k=0}^{k_r} I_k \right) \in \{0,1\}, \] 
so in this case the same proof strategy gives $\asp F = \max\{h,\uasp P\}$. Therefore by Lemma~\ref{l:countable} and \cite{Fraser2019-3}, $\asp F = \uasp P = \uasp G$, as required. 
\end{proof}

In particular, different possible forms for the Assouad spectrum of the limit set include being concave on open intervals, being a non-trivial differentiable function on $(0,1)$, or being non-differentiable on every open interval, see~\cite{Rutar2022assouad}. 
However, these exotic behaviours are coming from the set of fixed points rather than from the dynamical structure of the limit set. 
After submitting this paper, the authors in collaboration with Kolossv\'ary and Rutar~\cite{BanajiPreprintgl} have shown that for a class of finitely generated self-affine sets called Gatzouras--Lalley carpets, the Assouad spectrum can also have interesting behaviour, such as being concave on open intervals, or being a non-trivial differentiable function on $(0,1)$, or having many phase transitions. 

\section{Applications: continued fractions and parabolic IFSs}\label{s:ctdfracsect}
In this section we apply our methods to calculate the Assouad spectra of several interesting families of fractal sets. Sets of real or complex numbers which have continued fraction expansions with restricted entries are especially well-studied in the dimension theory of dynamical systems, see for instance~\cite{Banaji2021,Chousionis2020,Mauldin1999} and~\cite[Section~9.2]{Fraser2020book}.  
For a non-empty, proper subset $I \subset \N$, define 
\[ F_I \coloneqq \left\{ \, z \in (0,1) \setminus \mathbb{Q} : z = \frac{1}{b_1 + \frac{1}{b_2 + \frac{1}{\ddots}}}, b_n \in I \mbox{ for all } n \in \N \, \right\}. \]
Then $F_I$ is the limit set of the CIFS given by the inverse branches of the Gauss map $\mathcal{G} \colon [0,1] \to [0,1]$ (given by $\mathcal{G}(x) = \{ 1/x \}$ where $\{y\}$ denotes the fractional part of $y \geq 0$) corresponding to the elements of $I$. 
\begin{lma}\label{l:ctdfraccifs}
Working in $\mathbb{R}$, letting $X \coloneqq [0,1]$ and $V \coloneqq (-1/8,9/8)$, 
\begin{itemize}
\item\label{i:ctdfracnot1} If $1 \notin I$ then $\{ \, S_b(x) \coloneqq 1/(b+x) : b \in I \, \}$ is a CIFS with limit set $F_I$. 
\item\label{i:ctdfrac1} If $1 \in I$ then $\{ \, S_b(x) \coloneqq 1/(b+x) : b \in I, b \neq 1 \, \} \cup \left\{ \, S_{1b}(x) \coloneqq \frac{1}{b+\frac{1}{1+x}} : b \in I \, \right\}$ is a CIFS with limit set $F_I$. 
\end{itemize}
\end{lma}
\begin{proof}
This is well known (see~\cite[page~4997]{Mauldin1999}, for example). 
\end{proof}
Lemma~\ref{l:ctdfraccifs} shows why our general results can be applied in this setting, and is one of the reasons why we proved the general bounds above for conformal contractions rather than just for similarities. Recall that by~\cite[Theorem~3.15]{Mauldin1996}, the Hausdorff dimension $h$ can be determined by the topological pressure function.

We first show that subsets of $\N$ which satisfy an asymptotic condition give rise to continued fraction sets whose Assouad spectrum satisfies the three-parameter form and attains the upper bounds of Theorem~\ref{t:mainasp} and Corollary~\ref{c:special}. 

\begin{prop}\label{t:ctdfracdense}
If $I \subseteq \N$ satisfies 
\begin{equation}\label{e:fulldensity}
\limsup_{N \to \infty} \frac{\log \# (I \cap [1,N])}{\log N} = 1
\end{equation}
then for all $\theta \in (0,1)$ we have 
\[ \asp F_I = \left\{\begin{array}{lr}
        h + \frac{\theta}{(1-\theta)}(1-h), & \text{for } 0\leq \theta < 1/2 \\
        1, & \text{for } 1/2 \leq \theta \leq 1
        \end{array}\right. .
\]
\end{prop}%

\begin{proof}
If $I \subseteq \N$ satisfies the condition~\eqref{e:fulldensity} then by the proof of~\cite[Lemma~14.1.4]{Fraser2020book}, 
\[ \asp \{ \, 1/b : b \in I \, \} = \min\left\{\frac{1}{2(1-\theta)},1\right\} \]
for all $\theta \in [0,1]$. This will equal the Assouad spectrum of the set of fixed points of the contractions comprising the CIFS from Lemma~\ref{l:ctdfraccifs} even if $1 \in I$, because the Assouad spectrum is finitely stable and unchanged under bi-Lipschitz transformations. 
Fix an integer $k \geq 10$. By~\eqref{e:fulldensity} there exists a sequence of integers $M_1 < M_2 < \cdots$ such that $M_{n+1} > M_n^{k/(k-2)}$ and $\# (I \cap [1,M_n^{k/(k-2)}]) \geq (M_n^{k/(k-2)})^{1-1/k}$ for all $n \in \N$. %
Then $\# (I \cap [M_n,M_n^{k/(k-2)}]) \gtrsim M_n^{\frac{k-1}{k-2}}$, with the implicit constant independent of $n$. 
Now 
\[ \sum_{b \in I} (b^{-2})^{\frac{k-1}{2k}} \geq \sum_{n=1}^{\infty} \sum_{\substack{b \in I \\ M_n \leq b \leq M_n^{k/(k-2)}}} (b^{-2})^{\frac{k-1}{2k}} \gtrsim \sum_{n=1}^{\infty} M_n^{\frac{k-1}{k-2}} ((M_n^{k/(k-2)})^{-2})^{\frac{k-1}{2k}} = \infty . \]
The finiteness parameter (recall~\eqref{e:finiteness}) of the system therefore satisfies $\Theta \geq \frac{k-1}{2k}$. Letting $k \to \infty$ shows that $\Theta \geq 1/2$, so $\dim_{\mathrm B} F = h \geq 1/2$. %
Fix $s<h$ and $\theta \in (0,1/2)$. By Lemma~\ref{l:assouadgeo}, 
\begin{align*}
 N_{M_n^{-1/\theta}} (F \cap [0,M_n^{-1}]) &\geq N_{M_n^{-1/\theta}} (F \cap [M_n^{-k/(k-2)},M_n^{-1}]) \\
 &\gtrsim M_n^{\frac{k-1}{k-2}} \left( \frac{M_n^{-2k/(k-2)}}{M_n^{1/\theta}} \right)^s \\
&= M_n^{-(1-1/\theta)\left( \frac{s(1-2k\theta/(k-2)) + \theta(k-1)/(k-2)}{1-\theta} \right)}. 
\end{align*}
The result follows upon letting $s \to h^-$ and $k \to \infty$. 
\end{proof}

In Propositions~\ref{p:ctdspaced} and~\ref{p:ctdclustered} we consider the Assouad spectra of continued fraction sets for particular families of infinite subsets of $\N$. In Proposition~\ref{p:ctdspaced}, the numbers that we allow to be included in the continued fraction expansions are eventually spaced apart like $\{ \, \lfloor n^p \rfloor : n \in \N \, \}$ (similar to~\cite[Example~4.4]{Banaji2021} for the intermediate dimensions). However, we need to allow an arbitrary choice for the first finitely many digits of $I$ to ensure that the Hausdorff dimension can be in the range that yields non-trivial behaviour for the Assouad spectrum (such as two phase transitions).

\begin{prop}\label{p:ctdspaced}
Fix $p \in (1,\infty)$. Assume that $I \subseteq \N$ is such that the symmetric difference of $I$ and $\{ \, \lfloor n^p \rfloor : n \in \N \, \}$ is finite and that $h \in (1/(p+1),1/p)$. Then 
\[ \asp F_I = 
\left\{\begin{array}{lr}
        h + \frac{\theta}{p(1-\theta)}(1-ph), & \text{for } 0\leq \theta \leq \frac{(h+hp-1)p}{(1+p)(2ph-1)}\\
        \frac{1}{(1+p)(1-\theta)}, & \text{for } \frac{(h+hp-1)p}{(1+p)(2ph-1)} < \theta < \frac{p}{1+p}\\
        1, & \text{for } \frac{p}{1+p}\leq \theta \leq 1
        \end{array}\right. .\]

In particular, $\asp F$ has two phase transitions and lies strictly between the bounds in Theorem~\ref{t:mainasp} for all $\theta \in (0,\frac{(h+hp-1)p}{(1+p)(2ph-1)})$. 
\end{prop}

\begin{proof}
We start with the appropriate CIFS from Lemma~\ref{l:ctdfraccifs} (depending on whether or not $1 \in I$). Although the OSC always holds, the separation condition in Theorem~\ref{t:asd} does not hold if $I$ contains consecutive digits. To deal with this, we use the same trick as in the proof of Theorem~\ref{t:mainasp}, namely to iterate finitely many times to form a new CIFS where each conformal copy of $F$ is very small and use an induction argument. The proof is now very similar to the case $t=2p$ of the proof of Theorem~\ref{t:sharp}~\eqref{i:sharpmiddle}, but with the additional technicality that the bound from Lemma~\ref{l:assouadgeo} needs to be used since we are in the self-conformal (rather than strictly self-similar) setting. The details are left to the reader. 
\end{proof}

In Proposition~\ref{p:ctdspaced}, if $h \in (1/(2p),1/(p+1)]$ then the bounds coincide at the Assouad spectrum of the set of fixed points (which equals $\asp F_p = \min\left\{ \frac{1}{(1+p)(1-\theta)},1\right\}$) by Corollary~\ref{c:special}. If, on the other hand, $h \in [1/p,1)$, %
 then the lower bounds are attained. 
The Assouad spectrum for the above sets for a certain choice of parameters is shown in green in Figure~\ref{f:sharp}.

In Proposition~\ref{p:ctdclustered}, the elements of $I$ are very clustered together (in contrast to Proposition~\ref{p:ctdspaced}), resulting in the upper bound from Theorem~\ref{t:mainasp} and the three-parameter form being satisfied.

\begin{prop}\label{p:ctdclustered}
Fix $\alpha \in (0,1)$ and let $I \subset \N$ be such that the symmetric difference of $I$ and $\N \cap \bigcup_{k=1}^\infty [2^k,2^k + 2^{k \alpha}]$ is finite. Then the phase transition is $\rho = 1-\alpha/2$, and 
\[ \asp F_I  = \left\{\begin{array}{lr}
        h + \frac{\alpha \theta}{(1-\theta)(2-\alpha)}(1-h), & \text{for } 0\leq \theta < \rho \\
        1, & \text{for } \rho \leq \theta < 1
        \end{array}\right. . \] 
\end{prop}

\begin{proof}
A direct calculation shows that $\dim_{\mathrm A}^{1-\alpha /2} P = 1$, and combining this with~\cite[Lemma~3.4.4]{Fraser2020book}) shows that $\asp P = \min\left\{\frac{\alpha}{2(1-\theta)},1\right\}$ for all $\theta \in (0,1)$. Moreover, $\dim_{\mathrm B} F = h \geq \alpha/2$, since the finiteness parameter is easily calculated to be $\Theta = \alpha /2$ in this case. Moreover, for all $\epsilon > 0$ and $k \in \N$ we have $2^{-k} - (2^{-k} + 2^{k\alpha})^{-1} \approx 2^{-k(2-\alpha)}$ with the implicit constant independent of $k$. Therefore by Lemma~\ref{l:assouadgeo}, 
\begin{align*}
N_{2^{-k(2-\alpha )/\theta}}(F_i \cap [(2^k - 2^{-k(2-\alpha)},2^{-k}]) &\gtrsim 2^{k\alpha} \left(\frac{2^{-2k}}{2^{-k(2-\alpha )/\theta}}\right)^{h-\epsilon} \\
&= 2^{k[-(2-\alpha)(1-\theta^{-1})(h + \frac{\alpha \theta}{(1-\theta)(2-\alpha)}(1-h)) - (\frac{2-\alpha }{\theta} -2)\epsilon ]}. 
\end{align*} 
Since $\epsilon$ was arbitrary, this completes the proof. 
\end{proof}

We now calculate the Assouad spectrum of a special set related to complex continued fraction expansions. Namely, %
define 
\[ F_{\N + \Z i} \coloneqq \left\{ \, z \in \mathbb{C} : z = \frac{1}{b_1 + \frac{1}{b_2 + \frac{1}{\ddots}}}, b_n \in \N + \Z i \mbox{ for all } n \in \N \, \right\}. \]
It is clear from~\cite[Section~6]{Mauldin1996} that 
\begin{equation}\label{e:complexcifs}
 \{ \, S_b(z) \coloneqq 1/(b+z) : b \in (\N + \Z i) \setminus \{1\} \, \} \cup \left\{ \, S_{1b}(z) \coloneqq \frac{1}{b+\frac{1}{1+z}} : b \in \N + \Z i \, \right\} 
 \end{equation}
is a CIFS with limit set $F_{\N + \Z i}$. %
Estimates for the Hausdorff dimension $h = \dim_{\mathrm H} F_{\N + \Z i}$ are given in~\cite[Section~6]{Mauldin1996} and~\cite{Gardner1983,Priyadarshi2016}. To our knowledge, the tightest bounds to date are $1.85574 \leq h \leq 1.85589$ from~\cite{Falk2018}. 
The following result states that the Assouad spectrum satisfies a formula in terms of the Hausdorff dimension, given by the upper bound from Theorem~\ref{t:mainasp}. 
\begin{prop}
We have  
\[ \asp F_{\N + \Z i} = \left\{\begin{array}{lr}
        h + \frac{\theta}{1-\theta}(2-h), & \text{for } 0 < \theta < 1/2 \\
        2, & \text{for } 1/2 \leq \theta < 1
        \end{array}\right. . \]
\end{prop}
\begin{proof}
Let $P$ be the set of fixed points of the CIFS in~\eqref{e:complexcifs}. A similar proof to the proof of the $p=1$ case of~\eqref{e:fpspectrum} in~\cite[Corollary~6.4]{Fraser2018-2} shows that $\asp P = \min\{(1-\theta)^{-1},2\}$. Fix $0 < \theta < 1/2$ and very small $\epsilon \in (0,1)$, and suppose $0 < R < 1$. Let $I_{R,\epsilon} \coloneqq \{ \, i \in \N + \Z i : R^{1+\epsilon} < z < R \mbox{ for all } z \in F_i \, \}$. Then $\# I_{R,\epsilon} \approx R^{-2}$, with the implicit constants depending on $\epsilon$ but not on $R$. Using the Koebe distortion theorem we have $|F_i| \gtrsim R^{2(1+\epsilon)}$ for all $i \in I_{R,\epsilon}$. %
Therefore by Lemma~\ref{l:assouadgeo}, %
\[ N_{R^{1/\theta}} (B(0,R) \cap F_{\N + \Z i}) \gtrsim \sum_{i \in I_{R,\epsilon}} N_{R^{1/\theta}} (F_i) \gtrsim R^{-2} \left( \frac{R^{2(1+\epsilon)}}{R^{1/\theta}}\right)^{h-\epsilon} = R^{(1-\theta^{-1})\left( f(\theta,1/2) - \frac{1-2\theta \epsilon}{1-\theta} \epsilon \right)}.\] 
Letting $\epsilon \to 0^+$ gives $\asp F_{\N + \Z i} \geq f(\theta,1/2)$, from which the result follows. 
\end{proof}
It would be possible to study the Assouad spectra of the limit sets of appropriately chosen subsystems of the CIFS in~\eqref{e:complexcifs}, but we will not pursue this.

Our results also have applications to the dimension theory of parabolic iterated function systems. In such a system, each map $S \colon X \to X$ still satisfies $||S(x) - S(y)|| < ||x-y||$ for all $x,y \in X$, but may contain a \emph{parabolic fixed point} $p \in X$, meaning that $S(p) = p$ but the derivative of $S$ (or an extension of $S$) at $p$ has norm 1. 
The theory of parabolic IFSs has been developed by Mauldin and Urbański in~\cite{Mauldin2000parabolic}, and they have also been studied in~\cite{Urbanski1996paraboliccantor,Mauldin2002parabolic},~\cite[Section~9.2]{Fraser2020book}, and many other works. 

Given a (possibly infinite) parabolic IFS as defined in~\cite[Section~2]{Mauldin2000parabolic}, one can associate an `induced' uniformly contracting infinite CIFS (see~\cite[Theorem~5.2]{Mauldin2000parabolic}). It is clear that if $F$ is the limit set of the parabolic IFS and $F^*$ is the limit set of the induced CIFS then $F^* \subseteq F$ with $F \setminus F^*$ countable, and $F$ and $F^*$ have the same closure. Therefore if $\dim$ is any of the notions of dimension mentioned in this paper then $\dim F = \dim F^*$. In particular,~\cite[Theorem~3.5]{Banaji2021} for the intermediate dimensions and Theorem~\ref{t:mainasp} for the Assouad spectrum can be applied directly to the induced system to give information about the corresponding dimension of $F$. %

As an example, we consider a finite parabolic IFS on the line with a single parabolic fixed point of Manneville--Pomeau type at 0. In this setting, the Hausdorff dimension (denoted $h$) equals the box dimension~\cite{Urbanski1996paraboliccantor} and the Assouad dimension is 1~\cite[Theorem~9.2.1]{Fraser2020book}. In Proposition~\ref{p:parabolic} we show that the Assouad spectrum attains the upper bound from Theorem~\ref{t:mainasp}. 
\begin{prop}\label{p:parabolic}
Fix $N \in \N$ and $\epsilon, q > 0$. Suppose $S_1, \ldots, S_N \colon [0,1] \to [0,1]$ satisfy $|S_i(x) - S_i(y)| < |x-y|$ for all $x,y \in [0,1]$ and extend to $C^{1+\epsilon}$ maps on $(-\epsilon, 1+\epsilon)$. Assume $S_2,\ldots,S_N$ are $\xi$-Lipschitz for some $\xi < 1$ and that $S_i((0,1)) \cap S_j((0,1)) = \varnothing$ whenever $i \neq j$. Finally, assume $S_1(0) = 0$ and 
\[ \frac{x - S_1(x)}{x^{1+q}} \to 1 \quad \mbox{as} \quad x \to 0^+.\]
 Let $F$ denote the limit set of this parabolic IFS. Then for $\theta \in (0,1)$, 
 \begin{equation}\label{e:parabolic}
 \asp F = \left\{\begin{array}{lr}
        h + \frac{q \theta}{1-\theta}(1-h), & \text{for } 0\leq \theta < \frac{1}{1+q} \\
        1, & \text{for } \frac{1}{1+q}\leq \theta < 1
        \end{array}\right. . 
        \end{equation}
\end{prop}
\begin{proof}
Let $P$ be the set of fixed points of the induced CIFS 
\[ \{ \, S_1^n \circ S_i : n \in \N, 2 \leq i \leq N \, \} \cup \{ S_2, \ldots, S_N\}.\] 
Estimates similar to~\cite[(4.2)--(4.4)]{Mauldin2002parabolic} show that $\asp P = \asp (\{ \, i^{-1/q} : i \in \N \, \})$. The result now follows by a similar argument to the proof of Theorem~\ref{t:sharp}~\eqref{i:sharpupper}, using estimates similar to~\cite[(4.1)]{Mauldin2002parabolic}. 
\end{proof}
More generally, one could consider a finite parabolic iterated function system that generates a `parabolic Cantor set' in the sense of~\cite{Urbanski1996paraboliccantor}. Then several of the maps $S_i$ can have a parabolic fixed point $p_i$ with local behaviour $S_i(x) = p_i + a_i (x-p_i)^{1+q_i} + o((x-p_i)^{1+q_i} )$, and the Assouad spectrum will take the form of~\eqref{e:parabolic} with $q = \max_i q_i$. 

Finally, we apply Proposition~\ref{p:parabolic} to calculate the Assouad spectrum of sets generated by backwards continued fraction expansions. For $I \subseteq \{2,3,4,\ldots,\}$ define 
\[ \mathcal{B}_I \coloneqq \left\{ \, z \in (0,1) \setminus \mathbb{Q} : z = 1 - \frac{1}{b_1 - \frac{1}{b_2 - \frac{1}{\ddots}}}, b_n \in I \mbox{ for all } n \in \N \, \right\}. \]
Recall that the R\'enyi map $\mathcal{R} \colon [0,1] \to [0,1]$ is defined by $\mathcal{R}(x) \coloneqq \left\{ \frac{1}{1-x}\right\}$. 
Then $\mathcal{B}_I$ is the limit set of the (possibly parabolic, possibly infinite) IFS consisting of the inverse branches of the R\'enyi map corresponding to the elements of $I$. 
\begin{cor}\label{c:backward}
If $I \subseteq \{2,3,4,\ldots,\}$ is finite with $2 \in I$ then 
\[ \asp \mathcal{B}_I = \left\{\begin{array}{lr}
        h + \frac{\theta}{1-\theta}(1-h), & \text{for } 0\leq \theta < 1/2 \\
        1, & \text{for } 1/2 \leq \theta < 1
        \end{array}\right. . \]
\end{cor}
\begin{proof}
The inverse branch of the R\'enyi map corresponding to $2 \in I$ is $x \mapsto x/(1+x)$, which has a parabolic fixed point at 0, and all other branches are uniformly contracting. Therefore using Taylor's theorem, Corollary~\ref{c:backward} follows from the $q=1$ case of Proposition~\ref{p:parabolic}. %
\end{proof}

\section{Further work}\label{s:further}

We conclude the paper with some potential directions for future research. 
While Theorem~\ref{t:sharp} shows that the bounds in Theorem~\ref{t:mainasp} are in a sense sharp, there are further questions that one could ask about the possible form of the Assouad spectrum of the limit set of a CIFS. In particular, suppose we fix a countable set such as $P = F_p \coloneqq \{ \, i^{-p} : i \in \N \, \}$ where $p \in (0,\infty)$, and fix $h \in (\ubd P, \qad P)$. 
If we consider only those CIFSs whose set of fixed points is precisely the set $P$, and vary the contraction ratios in such a way that the Hausdorff dimension of the limit set $F$ stays constant at $h$, then the bounds from Theorem~\ref{t:mainasp} will stay the same, but we expect it would be a delicate problem to give a precise description of how $\asp F$ can vary between those bounds. 

Throughout this paper, as in~\cite{Mauldin1996,Mauldin1999}, we always assume at least the Open Set Condition (OSC) (the separation condition from Definition~\ref{d:cifs}). 
In Section~\ref{s:asd}, we additionally assume that ${\overline{S_i}(V) \cap \overline{S_j}(V) = \varnothing}$ for all distinct $i,j \in I$. This is satisfied in many natural settings. 
For example, it holds for the family of CIFSs in Lemma~\ref{l:existcifstechnical} with $t > p+1$. 
Nonetheless, it is natural to ask whether the assumption is really needed, or whether the OSC suffices. 
The lower bound $\asd F \geq \max\{h,\asd P\}$ is immediate from Lemma~\ref{l:samewithinlevelasd} even without the assumption, but it is not obvious to us how to prove the other inequality without the assumption. 
\begin{ques}
In the statement of Theorem~\ref{t:asd}, can the assumption that ${\overline{S_i}(V) \cap \overline{S_j}(V) = \varnothing}$ for all distinct $i,j \in I$ be removed? 
\end{ques}
For attractors of certain finite IFSs, one can say something about the Assouad dimension even if the OSC is not assumed. Indeed, Fraser et al.~\cite{Fraser2015assouad} prove a dichotomy for the Assouad dimension of self-similar sets on the line: if a condition called the Weak Separation Property (WSP) holds then Hausdorff and Assouad dimension coincide, but if the WSP fails then the Assouad dimension is $1$. This result has a natural generalisation to self-conformal sets on the line~\cite{Angelevska2020conformal}, and the Assouad dimension of overlapping self-similar sets in higher dimensions has been studied in~\cite{Garcia2020assouad}. 
Ngai and Tong~\cite{Ngai2016} and Chu and Ngai~\cite{Chu2020} study the Hausdorff, box and packing dimensions of the limit sets of IIFSs with overlaps that do not satisfy the OSC but do satisfy suitable extensions of the WSP. It is therefore natural to ask (though we will not pursue this) what can be said about the Assouad type dimensions of the limit sets of infinite iterated function systems with overlaps that do not satisfy the OSC but perhaps satisfy weaker separation conditions such as those considered in~\cite{Ngai2016}.

Finally, as described in~\cite{Fraser2020book}, there are notions of dimension which are dual to the Assouad type dimensions, and describe the `thinnest' part of the set in question. 
The \emph{lower dimension} of a bounded set $F \subset \R^d$ with more than one point is defined by 
\begin{equation*}
      \dim_\mathrm{L} F = \sup\{ \,
      \lambda : \exists \, C>0\mbox{ such that } \forall x \in F, \forall \, 0<r<R\leq |F|, N_r(B(x,R)\cap F) \geq C(R/r)^\lambda \, \}. 
    \end{equation*}
    The \emph{lower spectrum} at $\theta \in (0,1)$ is 
    \begin{equation*}
      \dim^{\theta}_{\mathrm{L}} F = \sup\{ \, 
      \lambda : \exists \, C>0\mbox{ such that } \forall \, x \in F, \forall \, 0<R\leq |F|,  N_{R^{1/\theta}}(B(x,R)\cap F) \geq CR^{\lambda(1-1/\theta)} \, \}. 
    \end{equation*}
    These dimensions satisfy $\dim_{\mathrm L} F \leq \dim^{\theta}_{\mathrm{L}} F \leq \underline{\dim}_{\mathrm B} F$ for all $\theta$. 
    For infinitely generated self-conformal sets, these share some properties with the Assouad type dimensions. For instance, although there are sets for which the lower spectrum is not monotonic in $\theta$ (see \cite[Section~8]{Fraser2018-2}), for the limit set of a CIFS the function $\theta \mapsto \dim^{\theta}_{\mathrm{L}} F$ \emph{is} monotonic by a very similar argument to the proof of Lemma~\ref{l:monotoniclemma}. 
It would be possible to calculate formulae for the lower type dimensions of the sets in Lemma~\ref{l:existcifstechnical} using a direct covering argument and ideas from the proof of Theorem~\ref{t:sharp}, but we will not do this because this rather long calculation would not illuminate new phenomena within the scope of this paper, whose main focus is on the Assouad spectrum. 
One notable feature of the lower type dimensions in the context of infinite IFSs is described by Proposition~\ref{p:lowerzero}. 
\begin{prop}\label{p:lowerzero}
There exists a CIFS consisting of similarity maps whose limit set $F$ satisfies 
\[ \dim_{\mathrm L} F = \dim_{\mathrm L}^{\theta} F = 0 \qquad \mbox{ for all } \theta \in (0,1). \]\end{prop} 
\begin{proof}
For $i \geq 2$, let $S_i(x) \coloneqq 2^{-i} x + i^{-1}$. This clearly forms a CIFS on $[0,1]$. But for all $\theta \in (0,1)$, for all $i$ sufficiently large (depending on $\theta$), $F \cap (i^{-1} - 2^{-i \theta}/2, i^{-1} + 2^{-i \theta}/2) = S_i(F)$, so $N_{2^{-i}} (F \cap (i^{-1} - 2^{-i \theta}/2, i^{-1} + 2^{-i \theta}/2)) = 1$. 
Therefore $\dim^{\theta}_{\mathrm{L}} F = 0$, and so also $\dim_{\mathrm L} F = 0$. 
\end{proof}
This result is in stark contrast to the situation for finite IFSs. Indeed, limit sets of finite self-conformal IFSs with the OSC are known to be Ahlfors--David regular (see~\cite[Corollary~6.4.4]{Fraser2020book}), so all the dimensions considered in this paper coincide for such sets. Moreover, \cite[Theorem~6.3.1]{Fraser2020book} implies that if $F$ is the attractor of an arbitrary finite IFS of consisting bi-Lipschitz contractions, and $F$ is not a singleton, then $\dim_{\mathrm L}^{\theta} F > 0$ for some $\theta \in (0,1)$.

\section*{Acknowledgements}
AB thanks Mariusz Urbański for a helpful conversation. We thank Alex Rutar and two anonymous referees for helpful comments on versions of this paper. Both authors were financially supported by a Leverhulme Trust Research Project Grant (RPG-2019-034). JMF was also supported by an EPSRC Standard Grant (EP/R015104/1) and an RSE Sabbatical Research Grant (70249). 
Most of this work was completed while AB was JMF's PhD student at the University of St Andrews, but AB was also supported by an EPSRC New Investigators Award
(EP/W003880/1) while a postdoc at Loughborough University.

\section*{References}
\printbibliography[heading=none]%

\end{document}